\documentclass[11pt]{amsart}
\input xy
\xyoption{all}
\usepackage{amsmath,amsthm, amscd, amssymb, amsfonts}

\usepackage[usenames, dvipsnames]{xcolor}

\newcommand{\Hc}{{\mathcal H}}

\newcommand{\comp}{\mbox{comp}}
\newcommand{\coev}{\mbox{coev}}
\newcommand{\ev}{\mbox{ev}}

\newcommand{\otb}{{\overline{\otimes}}}

\newcommand{\Mo}{{\mathcal M}}
\newcommand{\No}{{\mathcal N}}
\newcommand{\Mod}{{\mathcal Mod}}

\newcommand{\camod}{{}_{\ca}\rm{Mod}}

\newcommand{\adj}{{\mathcal Ad}}

\newcommand{\ca}{{\mathcal C}}

\newcommand{\ot}{{\otimes}}

\newcommand{\op}{\rm{op}}

\newcommand{\Ec}{{\mathcal E}}

\newcommand{\wbc}{{\widetilde{\Bc}}}

\newcommand{\cha}{{\mathcal A}}

\newcommand{\ele}{{\mathcal L}}

\newcommand{\D}{{\mathcal D}}

\newcommand{\Do}{{\mathcal D}}

\newcommand{\Bc}{{\mathcal B}}

\newcommand{\Fc}{{\mathcal F}}
\newcommand{\Gc}{{\mathcal G}}

\newcommand{\rev}{\rm{rev}}

\newcommand{\ra}{\rm{ra}}
\newcommand{\la}{\rm{la}}
\newcommand{\cf}{\rm{CF}}

\newcommand{\ku}{{\Bbbk}}

\newcommand{\uno}{ \mathbf{1}}
\newcommand{\C}{{\mathcal C}}

\newcommand{\id}{\mbox{\rm id\,}}

\newcommand{\Id}{\mbox{\rm Id\,}}

\newcommand{\vect}{\mbox{\rm vect\,}}

\newcommand{\Rex}{\mbox{\rm Rex\,}}
\newcommand{\Fun}{\operatorname{Fun}}

\newcommand\Hom{\operatorname{Hom}}
\newcommand\uhom{\underline{\Hom}}

\newcommand{\End}{\operatorname{End}}

\theoremstyle{plain}

\numberwithin{equation}{section}

\newtheorem{teo}{Theorem}[section]

\newtheorem{lema}[teo]{Lemma}

\newtheorem{cor}[teo]{Corollary}

\newtheorem{prop}[teo]{Proposition}

\newtheorem{claim}{Claim}[section]

\theoremstyle{definition}

\newtheorem{defi}[teo]{Definition}

  \newtheorem{exa}[teo]{Example}

\theoremstyle{remark}

\newtheorem{rmk}[teo]{Remark}

\def\pf{\begin{proof}}

\def\epf{\end{proof}}

\theoremstyle{remark}

\begin{document}

\title[The adjoint algebra for 2-categories]
{The adjoint algebra for 2-categories}
\author[   Bortolussi and Mombelli  ]{ Noelia Bortolussi and Mart\'in Mombelli
 }

\address{Facultad de Matem\'atica, Astronom\'\i a y F\'\i sica
\newline \indent
Universidad Nacional de C\'ordoba
\newline
\indent CIEM -- CONICET
\newline \indent Medina Allende s/n
\newline
\indent (5000) Ciudad Universitaria, C\'ordoba, Argentina}
 \email{ bortolussinb@gmail.com, nbortolussi@famaf.unc.edu.ar 
 \newline \indent\emph{URL:} https://sites.google.com/view/noeliabortolussi/página-principal}
\email{martin10090@gmail.com, martin.mombelli@unc.edu.ar
\newline \indent\emph{URL:}\/ http://www.https://www.famaf.unc.edu.ar/$\sim$mombelli}

\begin{abstract} For any 0-cell $B$ in a 2-category $\Bc$ we introduce the notion of adjoint algebra $\adj_B$. This is an algebra in the center of $\Bc$. We prove that, if $\ca$ is a finite tensor category, this notion applied to the 2-category $\camod$ of $\ca$-module categories,  coincides with the one introduced by Shimizu \cite{Sh2}. As a consequence of this general approach, we obtain new results on the adjoint algebra for tensor categories.
\end{abstract}

\date{\today}
\maketitle

%\tableofcontents

\section*{Introduction}
In  \cite{Sh1},  the author introduces the notion of adjoint algebra $\cha_\ca$ and the space of class functions $\cf(\ca)$ for any finite tensor category $\ca$. The adjoint algebra is defined as the end $$\cha_\ca=  \int_{X\in \ca} X\ot X^*.$$
The object $\cha_\ca$ is in fact an algebra in the Drinfeld center $Z(\ca)$.
Both, the adjoint algebra and the space of class functions, are  interesting objects that generalize the well known adjoint representation and the character algebra of a finite group.
 In \cite{Sh2}, the author introduces the notion of adjoint algebra $\cha_\Mo$ and the space of class functions $\cf(\Mo)$ associated to  any left $\ca$-module category $\Mo$, generalizing the definitions given in \cite{Sh1}. 
 \smallbreak
 
The present paper is  born in  the search of a way to compute the adjoint algebra of a finite tensor category graded by a finite group. Let $G$ be a finite group, and $\Do=\oplus_{g\in G} \ca_g$ be a $G$-graded tensor category. In this case $\ca=\ca_1$ is a tensor subcategory of $\Do$, and for each $g\in G$, $\ca_g$ is an invertible $\ca$-bimodule category. Our goal is to relate the adjoint algebras $\cha_\Do$ and $\cha_\ca$. In principle, we did not have enough intuition nor tools to achieve this. We suspected that the algebra $\cha_\Do$  is related to $\oplus_{g\in G} \cha_{\ca_g},$ and each algebra $\cha_{\ca_g}$ is related to $\cha_\ca$. Since a direct approach to the computation of $\cha_\Do$  was not successful, we had to make a plan using different tools.
 
\medbreak
Our starting point is a result obtained in \cite{BGM}. In \emph{loc. cit.}   the notion of a group action on a 2-category $\Bc$ is introduced. For a group $G$ acting on $\Bc$, it is also introduced a new 2-category  $\Bc^G$, called the $G$-\emph{equivariantization}.  As a main example, if $\Do=\oplus_{g\in G} \ca_g$ is a $G$-graded tensor category, it is shown that the group $G$ acts on the 2-category $\camod$ of $\ca$-module categories and there is a 2-equivalence of 2-categories
\begin{equation}\label{g-equi-2cat}
    (\camod)^G \simeq  {}_\Do\text{Mod}.
\end{equation}
In this way, we can relate the tensor categories $\D$ and $\ca$ using  tools from the theory of 2-categories.
Our plan, to compute the adjoint algebra $\cha_\Do$, is the following:
\begin{itemize}
    \item Generalize the notion of adjoint algebra to the realm of 2-categories, such that when applied to $\camod$ it coincides with the notion introduced by Shimizu.\smallbreak
    
    \item Study how the adjoint algebra defined on 2-categories behaves under 2-equivalences.\smallbreak
    
    \item Apply the results obtained in the previous items, and use  the biequivalence \eqref{g-equi-2cat} to present some relation between $\cha_\Do$ and $\cha_\ca$.
\end{itemize}
In the present contribution  we  focused only on the first two steps. The last step will be developed in a subsequent paper.\smallbreak

The contents of this work are organized as follows. In Section \ref{Section:prelimin}, we discuss some preliminary notions and results on ends and coends in finite categories. In Section \ref{Section:tensorcat}  we collect the necessary material on finite tensor categories and their representations that will be needed. In Section \ref{SubSection:character algebra} we recall   the definition given in \cite{Sh2} of the adjoint algebra  $\cha_\Mo$ associated to a representation $\Mo$ of a finite tensor category $\ca$. In Section \ref{Section:2-cat} we begin by recalling  the basics of the theory of 2-categories, we recall the definition of pseudonatural transformations, pseudofunctors, and we also recall the definition of the center of a 2-category $\Bc$, which is a monoidal category $Z(\Bc)$. We also introduce the notion of a \emph{rigid} 2-category; a straightforward generalization of the notion of rigid monoidal category, this is a 2-category such that any 1-cell has left and right duals.  For any finite tensor category $\ca$, the rigid 2-category of (left) $\ca$-module categories, denoted by $\camod$ is developed thoroughly; in particular we prove that the center  $Z(\camod)$ is monoidally equivalent to $Z(\ca).$ This is a crucial result, since we would like to relate the adjoint algebra of $\camod$ and the one introduced by Shimizu. Finally, in Section \ref{Section:adj-2cat}, for any rigid 2-category $\Bc$, and any 0-cell $B\in \Bc$, we introduce an algebra $\adj_B$ in the center $Z(\Bc)$. The definition of this object and its product seem to be quite natural. There is a price to pay for this simplicity; the proof that, the adjoint algebras for  $\camod$ and for $\ca$ are isomorphic, is quite cumbersome. Finally, we show that if $\Fc:\Bc\to \wbc$ is a 2-equivalence  of 2-categories, this establishes a monoidal equivalence $\widehat{\Fc}:Z(\Bc)\to Z(\wbc)$, and for any 0-cell $B$ there is an isomorphism $\widehat{\Fc}(\adj_B)\simeq \adj_{\Fc(B)}$ of algebras. We apply this result to the 2-category $\camod$ of $\ca$-module categories to obtain some results on the adjoint algebra for tensor categories.

\section{Preliminaries}\label{Section:prelimin}

Throughout this paper $\ku$ will denote an algebraically closed field. All categories are assumed to be abelian $\ku$-linear, all functors are additive $\ku$-linear, and all vector spaces and algebras are assumed to be over $\ku$.  We denote by $\Rex(\Mo,\No)$ the category of right exact functors from $\Mo$ to $\No$. If $\Mo, \No$ are two  categories, and $F:\Mo\to\No$ is a functor, we will denote by $F^{\la}, F^{\ra}: \No\to \Mo$, its left adjoint, respectively right adjoint of $F$, if it exists. If $A$ is an algebra, we  denote by ${}_A\Mod$ (respectively $\Mod_A$) the category of finite dimensional left $A$-modules (respectively right $A$-modules).
\subsection{Finite categories}
A 
\emph{finite} category \cite{EO} is a category  equivalent to a category ${}_A\Mod$ for some finite dimensional algebra $A$. Equivalently, a  category is said to be finite if it satisfies the following conditions:
\begin{itemize}
\item it has finitely many isomorphism classes of simple objects;
 \item each simple
object $X$ has a projective cover $P(X)$; \item all $\Hom$ spaces
are finite-dimensional; \item each object has finite length.
\end{itemize}

\subsection{Ends and coends}

For later use, we will need some basic results on ends and coends. For reference, the reader can find \cite{Mac}, \cite{FSS}  helpful. Let $\ca$ and $\Do$ be categories.  A {\em dinatural transformation} $d: S \xrightarrow{..} T$  between functors $S, T: \ca^{\op} \times \ca \to \Do$  is a collection of morphisms in $\Do$
\begin{align*}
  d_{X}: S(X,X) \to T(X,X), \quad X \in \ca,
\end{align*}
such that for any morphism $f: X \to Y$ in $\ca$
\begin{equation}\label{dinaturality}
  T(\id_X, f) \circ d_X \circ S(f, \id_X) = T(f, \id_Y) \circ d_Y \circ S(\id_Y, f).
\end{equation}
An {\em end} of $S$ is a pair $(E, p)$ consisting of an object $E \in \Do$  and a dinatural transformation $p: E \xrightarrow{..} S$ satisfying the following universal property. Here the object $E$ is considered as a constant functor.
For any pair $(D, d)$ consisting of an object $D\in \Do$ and a dinatural transformation $d: D \xrightarrow{..} S$, there exists a \textit{unique} morphism $h: D\to E$ in $\Do$ such that 
$$d_X = p_X \circ h \quad \text{ for any } X \in \ca.$$ A {\em coend} of $S$ is  (the dual notion of an end)   a pair $(C, \pi)$ consisting of an object $C \in \Do$ and a dinatural transformation $\pi: S \xrightarrow{..} C$ with the following universal property. For any pair $(B, t)$, where $B\in \Do$ is an object and $t: S\xrightarrow{..} B$ is a dinatural transformation, there exists a {\em unique} morphism $h:C\to B$ such that $h\circ \pi_X=t_X$ for any $X \in \ca$.

The end  and  coend of the functor $S$ are denoted, respectively, as
$$
  \int_{X \in \mathcal{C}} S(X,X)
  \quad \text{and} \quad
  \int^{X \in \mathcal{C}} S(X,X). $$

 The next results are well-known. We present the proofs for completness sake and because  we will need them later.
 \begin{prop}\label{properties-end} Let $\ca, \Do, \ca', \Do'$ be categories. Assume that $S, T:\ca^{\op}\to \Do$ are functors. Let $G:\ca'\to \ca$ be an equivalence of categories, and $F:\Do\to \Do'$ a functor with left adjoint given by $H:\Do'\to \Do$. The following statements hold.
 \begin{itemize}
 \item[(i)] The objects $F(\int_{X \in \mathcal{C}} S(X,X)), \int_{X \in \mathcal{C}} F\circ S(X,X)$ are isomorphic.
 
 \item[(ii)] The objects $\int_{X \in \mathcal{C}} S(X,X), \int_{Y \in \ca'} S(G(Y),G(Y))$ are isomorphic.
 
 \item[(iii)] If $S\simeq T$, then the ends $\int_{X \in \mathcal{C}} S(X,X), \int_{X \in \mathcal{C}} T(X,X)$  are isomorphic.
 \end{itemize}
 \end{prop}  
  \pf (i).  Since $(H,F)$ is an adjunction, there are natural transformations $e:H\circ F\to \Id_\Do$, $c:\Id_{\Do'}\to F\circ H$ such that
 \begin{equation}\label{adjunct-din} \id_{F(Y)}= F(e_Y) c_{F(Y)}, \quad \id_{H(Z)}=e_{H(Z)} H(c_Z),
 \end{equation}
  for any $Z\in \Do'$. For any $Y\in \ca$ let $\pi_Y: \int_{X \in \mathcal{C}} S(X,X)\to S(Y,Y)$ be the dinatural transformations associated to this end. Then $F(\pi_Y):F(\int_{X \in \mathcal{C}} S(X,X)) \to F(S(Y,Y))$ are dinatural. Let us show that this object together with the dinatural transformations $F(\pi_Y)$ satisfy the universal property of the end.
  
Let $E\in \Do'$ be an object together with dinatural transformations $\xi_Y:E\to  F(S(Y,Y))$, $Y\in\ca$. Hence the composition 
$$ e_{S(Y,Y)}H(\xi_Y): H(E)\to S(Y,Y)$$
 are dinatural. Thus, there exists a map $h:H(E)\to \int_{X \in \mathcal{C}} S(X,X)$ such that $\pi_Y\circ h=e_{S(Y,Y)}H(\xi_Y)$ for any $Y\in\ca$. Define 
 $d=F(h)c_E$. Using \eqref{adjunct-din} one can see that 
 $F(\pi_Y) d=\xi_Y$. Whence $F(\int_{X \in \mathcal{C}} S(X,X))$ together with the dinatural transformations $F(\pi_Y)$ satisfy the universal property of the end.

(ii). Let $\xi_X:  \int_{X \in \mathcal{C}} S(X,X) \xrightarrow{..}  S(X,X)$, $\eta_Y: \int_{Y \in \mathcal{C}'} S(G(Y),G(Y))\xrightarrow{..} S(G(Y),G(Y)) $ be the associated dinatural transformations. Let $H:\ca\to \ca'$ be a quasi-inverse of $G$ and $\alpha: G\circ H\to \Id_\ca$ be a natural isomorphism. For any $X\in \ca$ define $\lambda_X=S(\alpha^{-1}_X,\alpha_X)\eta_{H(X)}$. Since $\lambda$ is a dinatural transformation, there exists a map $h: \int_{Y \in \mathcal{C}'} S(G(Y),G(Y))\to \int_{X \in \mathcal{C}} S(X,X)$ such that $\lambda_X=\xi_Xh$. Also, define $d_Y: \int_{X \in \mathcal{C}} S(X,X)\to  S(G(Y),G(Y))$ as $d_Y=\xi_{G(Y)}$. It follows that $d$ is a dinatural tranformation, therefore there exists a morphism $\widetilde{h}: \int_{X \in \mathcal{C}} S(X,X)\to  \int_{Y \in \mathcal{C}'} S(G(Y),G(Y))$ such that $d_Y=\eta_Y \widetilde{h}.$ We have
$$\xi_X h\widetilde{h}=S(\alpha^{-1}_X,\alpha_X)\eta_{H(X)}   \widetilde{h}= S(\alpha^{-1}_X,\alpha_X) \xi_{G(H(X))}=\xi_X.  $$ The last equality follows from the dinaturality of $\xi$. It follows, from the universal property of the end, that $h\widetilde{h}=\id$. In a similar way, one can prove that $\widetilde{h}h=\id$. Thus $h$ is an isomorphism.

(iii). Let $\alpha:S\to T$ be a natural isomorphism. Let  
$$\xi_X: \int_{X \in \mathcal{C}} S(X,X)  \xrightarrow{..}   S(X,X),\,\, \eta_X: \int_{X \in \mathcal{C}} T(X,X)  \xrightarrow{..} T(X,X),$$ be the corresponding dinatural transformations. For any $X\in \ca$ define $d_X= \alpha_{(X,X)} \xi_X$. It follows that $d$ is a dinatural transformation, hence there exists a morphism $h: \int_{X \in \mathcal{C}} S(X,X) \to \int_{X \in \mathcal{C}} T(X,X) $ such that $d_X= \eta_X\circ h$. In a similar way, one can construct a morphis $\widetilde{h}: \int_{X \in \mathcal{C}} T(X,X) \to \int_{X \in \mathcal{C}} S(X,X)$ such that $\alpha^{-1}_{(X,X)}\eta_X=\xi_X \widetilde{h}$. Combining both equalities one gets that $\xi_X \widetilde{h} h=\xi_X$. By the universal property of the end $\widetilde{h} h=\id$. In a similar way it can be proven that $h\widetilde{h}=\id$.
  \epf

 \section{Finite tensor categories}\label{Section:tensorcat}
  For basic notions on finite tensor categories we refer to \cite{EGNO}, \cite{EO}. Let $\C$ be a finite tensor category over $\ku$, that is a rigid monoidal category with simple unit object such that the underlying category is finite. 
  
  If $\ca$ is a tensor category with associativity constraint given by $a_{X,Y,Z}:(X\ot Y)\ot Z\to X\ot (Y\ot Z) $, we will denote by $\ca^{\rev}$, the tensor category whose underlying abelian category is $\ca$, with \textit{reverse} monoidal product $\ot^{\rev}:\ca\times \ca\to \ca$, $X\ot^{\rev }Y=Y\ot X$, and associativity constraints
 $$a^{\rev}_{X,Y,Z}:( X\ot^{\rev }Y) \ot^{\rev} Z\to  X\ot^{\rev }(Y \ot^{\rev} Z),$$
  $$a^{\rev}_{X,Y,Z}:= a^{-1}_{Z,Y,X},$$
  for any $X,Y, Z\in \ca$.
 \subsection{Representations of tensor categories} A left \emph{module} category over 
$\ca$ is a  finite  category $\Mo$ together with a $\ku$-bilinear 
bifunctor $\otb: \ca \times \Mo \to \Mo$, exact in each variable,  endowed with 
 natural associativity
and unit isomorphisms 
$$m_{X,Y,M}: (X\otimes Y)\otb M \to X \otb
(Y\otb M), \ \ \ell_M: \uno \otb M\to M.$$ These isomorphisms are subject to the following conditions:
\begin{equation}\label{left-modulecat1} m_{X, Y, Z\otb M}\; m_{X\otimes Y, Z,
M}= (\id_{X}\otb m_{Y,Z, M})\;  m_{X, Y\otimes Z, M}(a_{X,Y,Z}\otb\id_M),
\end{equation}
\begin{equation}\label{left-modulecat2} (\id_{X}\otb l_M)m_{X,{\bf
1} ,M}= r_X \otb \id_M,
\end{equation} for any $X, Y, Z\in\C$, $M\in\Mo.$ Here $a$ is the associativity constraint of $\C$.
Sometimes we shall also say  that $\Mo$ is a  $\ca$-\emph{module category}.

\medbreak

Let $\Mo$ and $\Mo'$ be a pair of $\C$-module categories. A\emph{ module functor} is a pair $(F,c)$, where  $F:\Mo\to\Mo'$  is a functor equipped with natural isomorphisms
$$c_{X,M}: F(X\otb M)\to
X\otb F(M),$$ $X\in  \ca$, $M\in \Mo$,  such that
for any $X, Y\in
\ca$, $M\in \Mo$:
\begin{align}\label{modfunctor1}
(\id_X \otb  c_{Y,M})c_{X,Y\otb M}F(m_{X,Y,M}) &=
m_{X,Y,F(M)}\, c_{X\otimes Y,M}
\\\label{modfunctor2}
\ell_{F(M)} \,c_{\uno ,M} &=F(\ell_{M}).
\end{align}

There is a composition
of module functors: if $\Mo''$ is another $\C$-module category and
$(G,d): \Mo' \to \Mo''$ is another module functor then the
composition
\begin{equation}\label{modfunctor-comp}
(G\circ F, e): \Mo \to \Mo'', \qquad  e_{X,M} = d_{X,F(M)}\circ
G(c_{X,M}),
\end{equation} is
also a module functor.

\smallbreak  

A \textit{natural module transformation} between  module functors $(F,c)$ and $(G,d)$ is a 
 natural transformation $\theta: F \to G$ such
that for any $X\in \ca$, $M\in \Mo$:
\begin{gather}
\label{modfunctor3} d_{X,M}\theta_{X\otb M} =
(\id_{X}\otb \theta_{M})c_{X,M}.
\end{gather}
 Two module functors $F, G$ are \emph{equivalent} if there exists a natural module isomorphism
$\theta:F \to G$. We denote by $\Fun_{\ca}(\Mo, \Mo')$ the category whose
objects are module functors $(F, c)$ from $\Mo$ to $\Mo'$ and arrows are module natural transformations. 

\smallbreak
Two $\C$-modules $\Mo$ and $\Mo'$ are {\em equivalent} if there exist module functors $F:\Mo\to
\Mo'$, $G:\Mo'\to \Mo$, and natural module isomorphisms
$\Id_{\Mo'} \to F\circ G$, $\Id_{\Mo} \to G\circ F$.

A module is
{\em indecomposable} if it is not equivalent to a direct sum of
two non trivial modules. Recall from \cite{EO}, that  a
module $\Mo$ is \emph{exact} if   for any
projective object
$P\in \ca$ the object $P\otb M$ is projective in $\Mo$, for all
$M\in\Mo$. It is known that if $\Mo$ is an exact module category, then any module functor $F:\Mo\to \No$ is exact. If $\Mo$ is an exact indecomposable module category over $\ca$, the dual category $\ca^*_\Mo=\End_\ca(\Mo)$ is a finite tensor category \cite{EO}. The tensor product is the composition of module functors.

 A \emph{right module category} over $\ca$
 is a finite  category $\Mo$ equipped with an exact
bifunctor $\otb:  \Mo\times  \ca\to \Mo$ and natural   isomorphisms 
$$\widetilde{m}_{M, X,Y}: M\otb (X\ot Y)\to (M\otb X) \otb Y, \quad r_M:M\otb \uno\to M$$ such that
\begin{equation}\label{right-modulecat1} \widetilde{m}_{M\otb X, Y ,Z }\; \widetilde{m}_{M,X ,Y\ot Z } (\id_M \otb a_{X,Y,Z})=
(\widetilde{m}_{M,X , Y}\ot \id_Z)\, \widetilde{m}_{M,X\ot Y ,Z },
\end{equation}
\begin{equation}\label{right-modulecat2} (r_M\otb \id_X)  \widetilde{m}_{M,\uno, X}= \id_M\otb l_X.
\end{equation}

If $\Mo,  \Mo'$ are right $\ca$-modules, a module functor from $\Mo$ to $  \Mo'$ is a pair $(T, d)$ where
$T:\Mo \to \Mo'$ is a  functor and $d_{M,X}:T(M\otb X)\to T(M)\otb X$ are natural  isomorphisms
such that for any $X, Y\in
\ca$, $M\in \Mo$:
\begin{align}\label{modfunctor11}
( d_{M,X}\ot \id_Y)d_{M\otb X, Y}T(m_{M, X, Y}) &=
m_{T(M), X,Y}\, d_{M, X\ot Y},
\\\label{modfunctor21}
r_{T(M)} \,d_{ M,\uno} &=T(r_{M}).
\end{align}

It is also well-known that if $F:\Mo\to \No$ is a $\ca$-module functor then its right and left adjoint are also $\ca$-module functors, see for example \cite[Lemma 2.11]{DSS}.

\subsection{Bimodule categories}

Assume that $\ca, \Do$ are finite  tensor categories. A $(\ca, \Do)-$\emph{bimodule category}  is an abelian category $\Mo$  with left $\ca$-module category
 and right $\Do$-module category  structure
equipped with natural
isomorphisms $$\{\gamma_{X,M,Y}:(X\otb M) \otb Y\to X\otb (M\otb Y), X\in\ca, Y\in\Do, M\in \Mo\}$$ satisfying 
certain axioms. For details the reader is referred to \cite{Gr}.

If $\Mo$ is a   right $\ca$-module category then the opposite category $\Mo^{\op}$ has a left $\ca$-action given by 
$\triangleleft:\ca\times\Mo^{\op} \to \Mo^{\op},$ $(X, M)\mapsto  M\triangleleft\, X^*.$ Similarly, if $\Mo$ is a left  $\ca$-module category then $\Mo^{\op}$ has structure of right $\ca$-module category, with action given by 
$\Mo^{\op}\times \ca\to  \Mo^{\op},$  $(M,X)\mapsto X^*\triangleright M$.
If  $\Mo$ is a $(\ca,\Do)$-bimodule category then $\Mo^{\op}$ is a 
$(\Do,\ca)$-bimodule category. 

Assume $\Ec$ is another finite tensor category. Let $\Mo$ be a $(\Do, \ca)-$bimodule category and $\No$ is a  $(\ca,\Ec)$-bimodule category.  Then the Deligne tensor product $\Mo\boxtimes_\ca\No$ is a left $(\Do,\Ec)$-bimodule category. More details on Deligne's tensor product can be found in \cite{Gr}.

The following definition appeared in \cite{ENO}.
\begin{defi} A $(\Do, \ca)-$bimodule category $\Mo$ is \textit{invertible} if there exists a $(\ca,\Do)$-bimodule category $\No$ such that
$$ \Mo\boxtimes_\ca\No \simeq \Do, \quad  \No\boxtimes_\Do\Mo\simeq \ca,$$
as bimodule categories.
\end{defi}
\begin{prop}\label{funct-mod}\cite[Thm. 3.20]{Gr}  If $\Mo, \No$ are left $\ca$-module categories, there exists an equivalence of categories
\begin{equation}\label{equival-tens-prod-of-modcat}\Mo^{\op}\boxtimes_\ca\No \simeq \Fun_\ca(\Mo, \No).\qed
\end{equation}
\end{prop}

\subsection{The internal Hom}\label{subsection:internal hom} Let $\ca$ be a  tensor category and $\Mo$ be  a left $\C$-module category. For any pair of objects $M, N\in\Mo$, the \emph{internal Hom} is an object $\uhom(M,N)\in \C$ representing the functor $\Hom_{\Mo}(-\otb M,N):\ca\to \vect_\ku$. This means that there are natural isomorphisms, one the inverse of each other,
\begin{equation}\label{Hom-interno}\begin{split}\phi^X_{M,N}:\Hom_{\ca}(X,\uhom(M,N))\to \Hom_{\Mo}(X\otb M,N), \\
\psi^X_{M,N}:\Hom_{\Mo}(X\otb M,N)\to \Hom_{\ca}(X,\uhom(M,N)),
\end{split}
\end{equation}
 for all $M, N\in \Mo$, $X\in\ca$. If $N, \widetilde{N}\in \Mo$, and $f:N\to  \widetilde{N}$ is a morphism, naturality of $\phi$ implies that diagramm
\begin{equation*}
\xymatrix{
\Hom_{\ca}(X,\uhom(M,N))\ar[d]^{\beta \mapsto \uhom(M,f)\beta}\ar[rr]^{\phi^X_{M,N}}&& \Hom_{\Mo}(X\otb M,N) \ar[d]_{\alpha\mapsto f\alpha} \\
\Hom_{\ca}(X,\uhom(M,\widetilde{N})) \ar[rr]^{\phi^X_{M,\widetilde{N}}}&& \Hom_{\Mo}(X\otb M,\widetilde{N}) }
\end{equation*}
commutes. That is
\begin{equation}\label{nat-phi-oneside} f \phi^X_{M,N}(\beta)= \phi^X_{M,\widetilde{N}}(\uhom(M,f)\beta).
\end{equation}
If $\widetilde{X}\in\ca$ and $h:X\to \widetilde{X}$ then, the naturality of $\phi$ implies that the diagram
\begin{equation*}
\xymatrix{
\Hom_{\ca}(\widetilde{X},\uhom(M,N))\ar[d]^{\alpha \mapsto \alpha h}\ar[rr]^{\phi^{\widetilde{X}}_{M,N}}&& \Hom_{\Mo}(\widetilde{X}\otb M,N) \ar[d]_{\alpha\mapsto \alpha (h\otb \id_M)} \\
\Hom_{\ca}(X,\uhom(M,N)) \ar[rr]^{\phi^X_{M,N}}&& \Hom_{\Mo}(X\otb M,N) }
\end{equation*}
commutes. That is
\begin{equation}\label{nat-phi-oneside2}   \phi^X_{M,N}(\alpha)(h\otb \id_M)= \phi^X_{M,N}(\alpha h),
\end{equation}
for any $\alpha\in \Hom_{\ca}(\widetilde{X},\uhom(M,N))$. Also, the naturality of $\psi$ implies that for any $X, \widetilde{X}\in\ca$, and any morphism $\gamma:X\to \widetilde{X}$ the diagramm
\begin{equation*}
\xymatrix{
 \Hom_{\Mo}(\widetilde{X}\otb M,N)\ar[d]^{\alpha \mapsto \alpha(\gamma\otb\id_M)}\ar[rr]^{\psi^{\widetilde{X}}_{M,N}}&&\Hom_{\ca}(\widetilde{X},\uhom(M,N))  \ar[d]_{\alpha\mapsto \alpha\gamma} \\
 \Hom_{\Mo}(X\otb M,N) \ar[rr]^{\psi^X_{M,N}}&&\Hom_{\ca}(X,\uhom(M,N))}
\end{equation*}
commutes. That is
\begin{equation}\label{nat-psi-oneside}\psi^X_{M,N}(\alpha(\gamma\otb\id_M))=  \psi^{\widetilde{X}}_{M,N}(\alpha) \gamma,
\end{equation}
For any $\alpha\in \Hom_{\Mo}(\widetilde{X}\otb M,N).$
If $X\in \ca$, $M, N\in \Mo$, define 
$$\coev^\Mo_{X, M}:X\to  \uhom(M,X\otb M), \quad \ev^\Mo_{M, N}:\uhom(M,N)\otb M\to N,$$
$$\coev^\Mo_{X, M}=\psi^X_{M,X\otb M}(\id_{X\otb M}),  \quad\ev^\Mo_{M, N}= \phi^{\uhom(M,N)}_{M,N}(\id_{\uhom(M,N)}).$$

Define also $f_M=\ev^\Mo_{M,M}(\id_{\uhom(M,M)}\otb \ev^\Mo_{M, M})$, and
\begin{equation}\label{compositionM}
\comp^\Mo_M:\uhom(M,M)\ot \uhom(M,M)\to \uhom(M,M),
\end{equation}
$$\comp^\Mo_M=\psi^{\uhom(M,M)\ot \uhom(M,M)}_{M,M}(f_M). $$
It is known, see \cite{EO}, that $\uhom(M,M)$ is an algebra in the category $\ca$ with product given by $\comp^\Mo_M.$

For any $M\in \Mo$ denote by $R_M:\ca\to \Mo$ the functor $R_M(X)=X\otb M$. This is a $\ca$-module functor. Its right adjoint $R^{ra}_M: \Mo\to \ca$ is then $ R^{\ra}_M(N)=\uhom(M,N)$. Since the functor $R_M$ is a module functor, then so is $ R^{\ra}_M$.  We denote by
\begin{equation}\label{mod-struct-a}
  \mathfrak{a}_{X,M,N}: \uhom(M, X\otb N) \to X\ot \uhom(M,  N)
\end{equation}
the left $\ca$-module structure of $R^{\ra}_M$.

Let $ \mathfrak{b}^1_{X,M,N}:\uhom(X\otb M, N)\to \uhom(M,N)\ot X^* $  
be the isomorphisms induced by the natural isomorphisms
\begin{gather*}
  \Hom_\ca(Z, \uhom(X \otb M, N))
  \simeq \Hom_{\mathcal{M}}(Z \otb( X \otb M), N)\\ \simeq \Hom_{\mathcal{M}}((Z \ot X)\otb M, N)
  \simeq \Hom_{\mathcal{C}}(Z \otimes X, \uhom(M, N))\\
  \simeq \Hom_{\mathcal{C}}(Z, \uhom(M, N) \otimes X^*),
\end{gather*}
for any $X, Z\in \ca$, $M, N\in \Mo$. Define also
$$ \mathfrak{b}_{X,M,N}:\uhom(X\otb M, N)\ot X\to \uhom(M,N)$$ 
as the composition
$$\mathfrak{b}_{X,M,N}=(\id\ot \ev_X)  (\mathfrak{b}^1_{X,M,N}\ot \id_X).$$
\subsection{The Drinfeld center} Assume that $\ca$ is a finite tensor category with associativity constraint given by $a$. The Drinfeld center of $\ca$ is the category $Z(\ca)$ consisting of pairs $(V,\sigma)$, where $V\in \ca$, and $\sigma_X:V\ot X\to X\ot V$ is a family of natural isomorphisms such that 
\begin{equation}\label{braid-cent0} l_V\sigma_{\uno}=r_V,
\sigma_{X\ot Y}= a^{-1}_{X,Y,V}(\id_{X}\ot \sigma_Y)a_{X,V,Y}(\sigma_X\ot \id_{Y})a^{-1}_{V,X,Y},
\end{equation}
for any $X, Y\in \ca$. If $(V,\sigma)$, $(W,\tau)$ are objects in the center, a morphism $f:
(V,\sigma)\to (W,\tau)$ is a morphism $f:V\to W$ in $\ca$ such that
\begin{equation}\label{0braid-cent1} (\id_X\ot f)\sigma_X=\tau_X (f\ot \id_X),
\end{equation}
for any $X\in\ca$. For a description of the braided monoidal structure of $Z(\ca)$ the reader is referred to \cite{Ka}. The categories $Z(\ca)^{\rev} $ and $Z(\ca^{\rev} )$ are monoidaly equivalent. The functor 
\begin{equation}\label{monoidal-center-rev}
Z(\ca)^{\rev}  \to Z(\ca^{\rev}), \quad (V,\sigma)\mapsto (V,\sigma^{-1})
\end{equation}
is a monoidal equivalence.
\section{The character algebra  for tensor categories }\label{SubSection:character algebra}

Let $\ca$ be a finite tensor category, and let $\Mo$ be an exact indecomposable left module category over $\ca$. We will recall the definition of the adjoint algebra and the space of class functions of $\Mo$ introduced by Shimizu \cite{Sh2}.

Let  $\rho_\Mo:\ca\to \Rex(\Mo),$ $\rho_\Mo(X)(M)=X\otb M,  X\in \ca, M\in \Mo,$ be the action functor.
By \cite[Thm. 3.4]{Sh2}  the right adjoint of $\rho_\Mo$ is the functor $\rho_\Mo^{\ra}:  \Rex(\Mo)\to\ca$, such that for any $F\in \Rex(\Mo)$
$$ \rho_\Mo^{\ra}(F)= \int_{M\in \Mo} \uhom(M, F(M)).$$

The \textit{adjoint algebra} of the module category $\Mo$ is defined as $$\cha_\Mo:=\rho_\Mo^{\ra}(\Id_\Mo)=\int_{M\in \Mo} \uhom(M, M)\in \ca.$$
Assume that $\pi^\Mo_M:\cha_\Mo\xrightarrow{ .. } \uhom(M, M)$ are the  dinatural transformations of the end $\int_{M\in \Mo} \uhom(M, M)$. Turns out that $\cha_\Mo$ is in fact an algebra object in the center $Z(\ca)$. The half braiding of $\cha_\Mo$ is  $\sigma^\Mo_X: \cha_\Mo\ot X\to X\ot \cha_\Mo$. It is defined as the unique isomorphism such that the diagram
 \begin{equation}\label{half-braidig-ch}
    \xymatrix@C=90pt@R=16pt{
      \cha_\Mo \otimes X
      \ar[r]^-{\pi^\Mo_{X \otb M} \otimes \id_X}
      \ar[dd]_{\sigma^{\Mo}_X}
      & \uhom(X \otb M,  X \otb M) \otimes X
      \ar[d]^{\mathfrak{b}_{X, M, X \otb M}} \\
      & \uhom(M,  X \otb M)
      \ar[d]^{ \mathfrak{a}_{X,M,M}} \\
      X \otimes \cha_\Mo
      \ar[r]^-{\id_X \otimes \pi^\Mo_M}
      & X \ot \uhom(M, M).
    }
  \end{equation}
is commutative. Recall from Section \ref{subsection:internal hom} the definition of the morphisms $\mathfrak{b}, \mathfrak{a}$.
It was explained in \cite[Subection 4.2]{Sh2} that the algebra structure of $\cha_\Mo$ is given as follows.  The product and the unit of $\cha_\Mo$ are 
$$m_\Mo:\cha_\Mo\ot \cha_\Mo\to \cha_\Mo, \quad u_\Mo:\uno\to \cha_\Mo,$$
defined to be the unique morphisms such that they satisfy
\begin{equation}\label{product-cha-M}
\begin{split}\pi^\Mo_M\circ m_\Mo= \comp^\Mo_{M}\circ (\pi^\Mo_M\ot \pi^\Mo_M),\\ \pi^\Mo_M\circ  u_\Mo= \coev^\Mo_{\uno, M},\end{split}
\end{equation}

for any $M\in \Mo$. For the definition of $\coev^\Mo$ and $ \comp^\Mo$ see Section \ref{subsection:internal hom}.

\medbreak
The \textit{adjoint algebra of the tensor category }$\ca$ is the algebra $\cha_\ca$ of the regular module category $\ca$.
\begin{defi}\label{definition:classfunct}\cite[Definition 5.1]{Sh2} The \textit{space of class functions} of $\Mo$ is $\cf(\Mo):=\Hom_{Z(\ca)}( \cha_\Mo,  \cha_\ca)$.
\end{defi}
\section{2-categories}\label{Section:2-cat}
 
We will first recall basic notions of the theory of 2-categories. 
For any 2-category $\Bc$, the class of \textit{0-cells}, will be denoted by $\Bc^0$. The composition in each hom-category $\Bc(A, B)$, is denoted by juxtaposition $fg$, while the symbol $\circ$ is used to denote the horizontal
composition functors $$\circ : \Bc(B, C) \times \Bc(A, B) \to \Bc(A, C).$$ For any 0-cell $A$ the identity 1-cell
is $I_A : A \to A$. For any 1-cell $X$ the identity will be denoted $\id_X$ or sometimes simply as $1_X$, when space saving measures are needed.

A 1-cell $X\in \Bc(A,B)$ is an \textit{equivalence} if there exists another 1-cell $Y\in  \Bc(B,A)$ such that 
$$X\circ Y\simeq I_B, \quad Y\circ X\simeq I_A.$$
Two 0-cells $A, B\in \Bc^0$ are \textit{equivalent}  if there exists a 1-cell equivalence $X\in  \Bc(A,B)$.

Assume $\wbc$ is another 2-category. A \textit{pseudofunctor}  $(F,\alpha):\Bc\to \wbc$, consists of a function $F:\Bc^0\to \wbc^0$, a family of functors $F_{A,B}:\Bc(A,B)\to \wbc(F(A),F(B)),$ for each $A, B\in \Bc^0$, and a collection of natural isomorphisms
$$\alpha_{A,B,C}:\circ^{F(A),F(B),F(C)}(F_{B,C}\times F_{A,B})\to F_{A,C} \circ^{A,B,C},$$
$$\phi_A: I_{F(A)}\to F_{A,A}(I_A),$$ such that
\begin{equation}\label{pseudofunc-axiom1}
\alpha_{X\circ Y,Z}(\alpha_{X,Y}\circ \id_{F(Z)})= \alpha_{X,Y\circ Z} (\id_{F(X)}\circ \alpha_{Y,Z}),
\end{equation}
\begin{equation}\label{pseudofunc-axiom2}\phi_B \circ \id_{F(X)}=\alpha_{I_B,X},\quad \id_{F(X)} \circ \phi_A=\alpha_{X,I_A}, \end{equation}
for any 0-cells $A, B, C$ and 1-cells $X, Y, Z$. A pseudofunctor is said to be a \textit{2-functor} if $\alpha$ and $\phi$ are identities.

Assume that  $(F,\alpha)$, $(G,\alpha')$ are pseudofunctors. A\textit{ pseudonatural transformation} $\chi:F\to G$
consists of a family of 1-cells $\chi^0_A: F(A)\to G(A)$, $A\in \operatorname{Obj}(\Bc)$ and isomorphisms 2-cells
$$\chi_X:G(X)\circ \chi^0_A\longrightarrow \chi^0_B\circ F(X), $$
natural in $X\in \Bc(A,B)$, such that
for any 1-cells  $X\in\Bc(B,C), Y\in \Bc(A,B)$
\begin{equation}\label{pseudonat-def1}
(\chi_X\circ \id_{F(Y)}) (\id_{G(X)}\circ \chi_Y)(\alpha'_{X,Y}\circ \id_{\chi^0_A})= 
(\id_{\chi^0_C}\circ \alpha_{X,Y})\chi_{X\circ Y},
\end{equation}
\begin{equation}\label{pseudonat-def2}
\chi_{I_A}(\id_{\chi^0_A}\circ \phi_A)=\phi'_A\circ\id_{\chi^0_A}.
\end{equation}
If $F:\Bc\to \wbc$ is a pseudofunctor, the identity pseudonatural transformation $\id:F\to F$ is defined as
$$(\id_F)^0_A=I_{F(A)},\,\, (\id_F)_X=\id_{F(X)},$$
for any 0-cells $A, B\in \Bc^0$, and any 1-cell $X\in \Bc(A,B)$.

If $\chi, \theta$ are pseudonatural transformations, a \textit{modification} $\omega:\chi\Rightarrow \theta$
consists of a family of 2-cells $\omega_A:\chi^0_A\to \theta^0_A$,  such that the diagrams
$$
\xymatrix{
G(X)\circ \chi^0_A  \ar[d]_{\operatorname{id}_{G(X)}\circ \omega_A} \ar[r]^{\chi_X} & \chi^0_B\circ F(X)  \ar[d]^{\omega_B\circ \operatorname{id}_{F(X)} }\\
G(X)\circ \theta^0_A \ar[r]^{\theta_X} & \theta^0_B\circ F(X) 
}
$$
commute for all $X\in \Bc(A,B)$.  Given pseudofunctors $F,G:\Bc\to \wbc$, we will denote $\operatorname{PseuNat}(F,G)$ the category where objects are pseudonatural transformations from $F$ to $G$ and arrows are modifications.

Two pseudonatural transformations
$(\eta, \eta^0), (\sigma, \sigma^0)$ are \emph{equivalent}, and we denote this by $(\eta, \eta^0) \sim(\sigma, \sigma^0)$,
if there exists an invertible modification $\gamma: (\eta, \eta^0)\to  (\sigma, \sigma^0)$. 

Assume that  $\Bc_i$, $i=1,2$ are 2-categories,  $F_i:\Bc_1\to \Bc_2$, $i=1,2,3$  are pseudofunctors 
and $(\sigma, \sigma^0):F_1\to  F_2$, $(\theta, \theta^0):F_2\to  F_3$  are pseudonatural transformations. The \emph{horizontal composition} $(\tau, \tau^0)=(\theta, \theta^0)\circ (\sigma, \sigma^0):F_1\to  F_3$ is the pseudonatural transformation given by
\begin{equation}\label{composition-pseudonat}\tau^0_A=\theta^0_A\circ  \sigma^0_A, \quad \tau_X=(\id_{\theta^0_B}\circ \sigma_X)(\theta_X\circ \id_{\sigma^0_A}),
\end{equation}
for any pair of 0-cells $A, B$ and any 1-cell $X\in \Bc_1(A,B)$. 

If $F:\Bc\to\Bc'$ y $G:\Bc\to\Bc'$ are pseudofunctors,  we say that a pseudonatural transformation $(\theta, \theta^0):F\to G$ is an  \emph{equivalence} if there exists  a pseudonatural transformation $(\theta', \theta'^0):G\to F$ such that $(\theta, \theta^0)\circ (\theta', \theta'^0)\sim\id_G$ and $(\theta', \theta'^0)\circ(\theta, \theta^0)\sim\id_F$. In such case,  we say that $F$ and $G$ are \emph{equivalent } and we denote it by  $F \sim G$.

We say that two 2-categories $\Bc$ y $\Bc'$ are \emph{biequivalent} if there exists pseudo\-functors
 $F:\Bc\to\Bc'$ and $G:\Bc'\to\Bc$ such that
$G\circ F \sim \Id_{\Bc}$, $F\circ G\sim \Id_{\Bc'}$. In such case, we say that  $F$ and $G$ are \emph{biequivalences}. It is well known that $F:\Bc\to\Bc'$ is a biequivalence if and only if $F_{A,B}$ are equivalences of categories for any 0-cells $A, B$, and $F$ is surjective (up to equivalence) on 0-cells.

\begin{exa}\label{exa2-cat-monoidal} If $\ca$ is a strict monoidal category, we denote by $\underline{\ca}$ the 2-category with a single 0-cell $\star$ and  $\underline{\ca}(\star, \star)=\ca$. 
The horizontal composition is given by the monoidal product of $\ca$. 
\end{exa}

 The next remark is a generalization of a result in the theory of tensor categories and it will be used later.
\begin{rmk}\label{interchange}
If $\Bc$ is a 2-category, let $A$ be a 0-cell and $X \in \Bc(A,A)$ a 1-cell. For any pair of 2-cells $f: X \rightarrow I_A$ and $g : I_A \rightarrow X$ we have 
$$\id_{I_A} \circ f = f = f \circ \id_{I_A} \quad \text{and} \quad  g f = (f \circ \id_X)(\id_X \circ g).$$ Indeed,
$$ gf = (\id_{I_A} \circ g)(f \circ \id_{I_A}) = (f \circ \id_X)(\id_X \circ g). $$
This implies that if $f:I_A\to X$ is an isomorphism 2-cell, then
$$ \id_X\circ f=f\circ \id_X.$$
\end{rmk}

\subsection{Finite 2-categories} We will introduce the notion of finite rigid 2-categories. This definition is an analogue of a finite rigid monoidal category. A similar definition appears in \cite{MM1, MM2}.

Let $\Bc$ be a 2-category and  $X\in\Bc(A,B)$ a 1-cell. A \textit{right dual} of $X$ is a 1-cell 
$X^*\in\Bc(B,A)$ equipped with 2-cells
$$ev_X:X^*\circ X\to I_A, \quad coev_X:I_B\to X\circ X^*,$$
such that the compositions
$$ X= I_B \circ X\xrightarrow{\;\;coev_X\circ \id\;\;}
X\circ X^*\circ X
\xrightarrow{\;\;\id_X\circ ev_X\;\;}X\circ I_A =X,
$$
$$
X^*= X^*\circ I_B
 \xrightarrow{\;\id\circ coev_X\;} X^*\circ X\circ X^* 
\xrightarrow{\;\; ev_X\circ \id\;\;}I_A \circ X^*
=X^*
$$
are the identities. Analogously, one can define the \textit{left dual} of $X$ as an object $ {}^*X\in\Bc(B,A)$ equipped with 2-cells
$$ev_X:X\circ  {}^*X\to I_B, \quad  coev_X:I_A\to {}^*X\circ X$$
satisfying similar axioms.
 We say that $\Bc$ is \textit{rigid}  if any 1-cell has right and left duals.

\begin{rmk} A rigid 2-category with a single 0-cell is a strict monoidal rigid category.
\end{rmk}

\begin{defi} A \textit{finite 2-category} is a rigid 2-category $\Bc$ such that $\Bc(A,B)$ is a finite category, for any 0-cells $A,B$, and $I_A\in \Bc(A,A)$ is a simple object for any 0-cell $A$.
\end{defi}

In the  next Lemma we collect some  basic results. The proof follows the same lines as those in the theory of tensor categories, and will be omitted. We only give an idea of the proof of one statement, since it will be needed later. 
\begin{lema}\label{dual-of-composition} Let $\Bc$ be a finite 2-category. The following statements hold.
\begin{itemize}
\item[(i)] For any 0-cells $A,B, C\in \Bc$, the functor
$$\circ: \Bc(A,B)\times \Bc(C,A)\to \Bc(C,B)$$
is exact in each variable.
\item[(ii)] For any pair  $X, Y$ of composable 1-cells, there are isomorphisms
$$  {}^*(X\circ Y)\simeq  {}^*Y\circ  {}^*X, \quad (X\circ Y)^*\simeq Y^*\circ X^*. $$

\item[(iii)] If $X\in \Bc(A,B)$ is an invertible 1-cell, with inverse $Y$, and isomorphisms $\alpha: X \circ Y \rightarrow I_B$ and $\beta: Y \circ X \rightarrow I_A$. Then $X$ is a left dual of $Y$, with evaluation and coevaluation given by 
$$ \coev_Y=\alpha^{-1}, \quad \ev_Y= \beta (\id_Y\circ \alpha\circ \id_X)(\id_{Y\circ X}\circ \beta^{-1}).$$
\end{itemize}

\end{lema}
\pf (ii).
We will give the first isomorphism. The second one is constructed similarly. Let be $\ev_{X\circ Y}: (X\circ Y)\circ {}^*(X\circ Y)\to I$, $\coev_X: I\to  {}^*X\circ X$, $\coev_Y: I\to  {}^*Y\circ Y$ the corresponding evaluation and coevaluations. The map $\phi:  {}^*(X\circ Y)\to  {}^*Y\circ  {}^*X $ defined as
\begin{equation}\label{iso-dualsl}\phi=\big(\id_{  {}^*Y\circ  {}^*X }\circ\ev_{X\circ Y} \big) \big(\id_{{}^*Y}\circ\coev_X\circ\id_{Y\circ {}^*(X\circ Y) }\big) \big(\coev_Y\circ \id_{ {}^*(X\circ Y)}\big)
\end{equation} is an isomorphism.

(iii).  Let us prove the rigidity axiom
$(\ev_Y \circ \id_Y)(\id_Y \circ \coev_Y)= \id_Y$. Starting from the left hand side
\begin{align*}
(\ev_Y \circ \id_Y)(& \id_Y \circ  \coev_Y)  = \\
& = (\beta \circ \id_Y)(\id_Y \circ \alpha \circ \id_{X \circ Y})(\id_{Y \circ X} \circ \beta^{-1} \circ \id_Y)(\id_Y \circ \alpha^{-1})\\
& = (\beta \circ \id_Y)(\id_{Y \circ X \circ Y} \circ \alpha)(\beta^{-1} \circ \id_{Y \circ X \circ Y})(\id_Y \circ \alpha^{-1})\\
& = (\beta \circ \id_Y)(\beta^{-1} \circ \id_Y)(\id_Y \circ \alpha)(\id_Y \circ \alpha^{-1}) \\
 & = \id_Y.
\end{align*}
The first equality is by the definition of $\ev_Y$ and $\coev_Y$, and the second equality is the consequence of apply the remark \ref{interchange} for $\alpha$ and $\alpha^{-1}$ and for $\beta$ and $\beta^{-1}$. In a similar way, one can prove the other axiom $(\id_{{}^*Y} \circ \ev_Y)(\coev_Y \circ \id_{{}^*Y}) = \id_{{}^*Y}$.
\epf

The next result relates the duals of pseudonatural equivalences. This result will be needed later.

\begin{lema}\label{dual-of-pseudonat-equi}   Let $\Bc, \widetilde{\Bc}$ be finite 2-categories. Let  $\Fc:\Bc\to \wbc$, $\Gc:\Bc\to \wbc$ be a pair of 2-functors, and let $\chi:\Fc\to \Gc$, $\tau:\Gc\to \Fc$ be pseudonatural equivalences, one the inverse of the other. For any 1-cell $X\in \Bc(A,B)$, ${}^*(\tau_X)=\chi_{{{}^*X}}$.
\end{lema}
\pf Since $\chi\circ \tau \sim \id$, there is an invertible modification $\omega: \chi\circ \tau \to \id$. Thus, we have isomorphisms $\omega_A: \chi^0_A \circ \tau^0_A \to I_A$, for any 0-cell $A$.
Using Lemma \ref{dual-of-composition} (iii) we have that $\chi^0_A$ is a left dual of $(\tau^0_A)$, with coevaluation and evaluation given by
$$\coev_{\tau^0_A}= \omega^{-1}_A, \,\,
\ev_{\tau^0_A}=\beta (\id_{\tau^0_A}\circ  \omega_A\circ \id_{\chi^0_A})(\id_{\tau^0_A\circ \chi^0_A}\circ \beta^{-1}),$$
With $\beta: \tau^0_A\circ \chi^0_A\to I$ some isomorphism, that we know  it exists since $\tau\circ \chi  \sim \id.$
For any 1-cell $X\in \Bc(A,B)$
$$\tau_X: \Fc(X)\circ \tau^0_A\to \tau^0_B \circ \Gc(X). $$
The naturality of $\tau$ implies that for any 1-cell $Y\in\Bc(A,B)$
\begin{equation}\label{natu-tau01}
\tau_{{}^*Y\circ Y} (\Fc(\coev_Y)\circ \id_{\tau^0_A})=\id_{\tau^0_B}\circ \coev_{\Gc(Y)}
\end{equation}
Using \eqref{pseudonat-def1} for $\tau$, equation \eqref{natu-tau01} implies that
\begin{equation}\label{natu-tau02}(\tau_{{}^*Y}\circ \id_{\Gc(Y)}) (\id_{{}^*\Fc(Y)}\circ \tau_Y) (\Fc(\coev_Y)\circ \id_{\tau^0_A})=\id_{\tau^0_B}\circ \coev_{\Gc(Y)}.
\end{equation}

Whence
\begin{equation}\label{comb0} (\id_{{}^*\Fc(Y)}\circ \tau_Y) (\Fc(\coev_Y)\circ \id_{\tau^0_A})=(\tau^{-1}_{{}^*Y}\circ \id_{\Gc(Y)})(\id_{\tau^0_B}\circ \coev_{\Gc(Y)}). \end{equation} 
Using that $\chi$ is the inverse of $\tau$ we get that
$$(\omega_B\circ \id_{\Gc(Y)})(\chi\circ \tau)_Y=(  \id_{\Gc(Y)}\circ \omega_A).$$
From this equation, and the definition of the composition of pseudonatural transformations \eqref{composition-pseudonat}, we obtain that
\begin{equation}\label{comb1} (\chi_X\circ \id_{\tau^0_A})(\id_{\Gc(X)}\circ \omega^{-1}_A)(\omega_B\circ\id_{\Gc(X)}) = \id_{\chi^0_B}\circ \tau^{-1}_X.
\end{equation} 
Combining  \eqref{comb0} and \eqref{comb1}, and using that ${}^*(\tau^0_A)=\chi^0_A$,
we obtain that
\begin{equation}\label{t0p-in} \begin{split} &(\id_{{}^* \tau^0_A\circ {}^*\Fc(Y)}\circ \tau_Y) (\id_{{}^* \tau^0_A}\circ \Fc(\coev_Y)\circ \id_{\tau^0_A})=\\
&=(\id_{{}^* \tau^0_A}\circ\tau^{-1}_{{}^*Y}\circ \id_{\Gc(Y)})(\id_{{}^* \tau^0_A\circ\tau^0_B}\circ \coev_{\Gc(Y)})\\
&=(\chi_{{}^*Y}\circ \id_{\tau^0_B\circ\Gc(Y)}) (\id_{\Gc({}^*Y)}\circ \omega^{-1}_B\circ \id_{\Gc(Y)})(\omega_A\circ\id)(\id_{{}^* \tau^0_A\circ\tau^0_B}\circ \coev_{\Gc(Y)}).
\end{split}
\end{equation} 
 For any 1-cell $Y\in\Bc(A,B)$ we have that ${}^*(\tau_Y)$ is equal to
 \begin{align*} &=(\id_{} \circ \ev_{\tau^0_B\circ \Gc(Y)} ) (\id_{{}^*\tau^0_A\circ {}^*\Fc(Y)}  \circ \tau_Y\circ \id_{ {}^*\Gc(Y)\circ {}^*\tau^0_B} )(\coev_{ \Fc(Y)\circ \tau^0_A} \circ \id_{} )\\
 &=  (\id_{{}^*\tau^0_A\circ {}^*\Fc(Y)} \circ \ev_{\tau^0_B} )(\id_{{}^*\tau^0_A\circ {}^*\Fc(Y)\circ\tau^0_B}\circ \ev_{\Gc(Y)}\circ \id_{{}^*\tau^0_B})\\ &(\id_{{}^*\tau^0_A\circ {}^*\Fc(Y)}  \circ \tau_Y\circ \id)
 ( (\id_{{}^*\tau^0_A}\circ \coev_{ \Fc(Y)}\circ \id_{\tau^0_A})) \coev_{\tau^0_A}\circ \id_{ {}^*\Gc(Y)\circ {}^*\tau^0_B})
\\
&=(\id_{{}^*\tau^0_A\circ {}^*\Fc(Y)} \circ \ev_{\tau^0_B} )(\id_{{}^*\tau^0_A\circ {}^*\Fc(Y)\circ\tau^0_B}\circ \ev_{\Gc(Y)}\circ \id_{{}^*\tau^0_B})\\
&(\chi_{{}^*Y}\circ \id_{\tau^0_B\circ\Gc(Y)\circ {}^*\Gc(Y)\circ {}^*\tau^0_B}) (\id_{\Gc({}^*Y)}\circ \omega^{-1}_B\circ \id_{\Gc(Y)\circ {}^*\Gc(Y)\circ {}^*\tau^0_B})(\omega_A\circ\id)\\
&(\id_{{}^* \tau^0_A\circ\tau^0_B}\circ \coev_{\Gc(Y)})(\coev_{\tau^0_A}\circ\id_{ {}^*\Gc(Y)\circ {}^*\tau^0_B})= \chi_{{}^*Y}
 \end{align*}
 The first equality is the definition of ${}^*(\tau_Y)$, the second equality  follows from the formula for $\ev_{X\circ Y}$ and $\coev_{X\circ Y}$. The third equality follows from \eqref{t0p-in}, and the last equality follows from the rigidity axioms.
\epf

\subsection{The 2-category of $\ca$-module categories}\label{example-modu-over-tensor}

 Let $\ca$ be a finite tensor category. We associate to $\ca$ diverse 2-categories that will be used throughout. Let $\camod$ be the \textit{2-category of representations of }$\ca$, that is defined as follows.  Its 0-cells are finite left $\ca$-module categories, 
if $\Mo, \No$ are  left $\ca$-module categories, then the category $\camod(\Mo, \No)=\Fun_\ca(\Mo, \No).$
The 2-category $\camod_e$ of indecomposable exact left $\ca$-module categories
is defined in a similar way as $\camod$, with 0-cells being the indecomposable left exact  
$\ca$-module categories.

\begin{lema}\label{C-mod-is-finite}  Let $\ca$ be a finite tensor category. The 2-category $\camod_e$ is a finite 2-category. 
\end{lema}
\pf For any indecomposable $\Mo$, the indentity functor $\Id_\Mo$ is a simple object. We  need to prove the existence of left and right duals. Let $\Mo, \No$ be exact $\ca$-module categories and $F:\Mo\to \No$ be a $\ca$-module functor. Then, $F$ is exact \cite[Prop. 3.11]{EO}. Then $F^*:\No\to \Mo$ is the left adjoint of $F$, and the left dual
$ {}^*F$ is the right adjoint of $F$. Both functors $F^*, {}^*F$ are $\ca$-module functors \cite[Lemma 2.11]{DSS}. The evaluation and coevaluation are given by the unit and counit of the adjunction. 
\epf

Let $\Mo\in\camod_e $, and $M\in \Mo$. According to Section \ref{subsection:internal hom} the right adjoint of the functor $R_M:\ca\to \Mo$, $R_M(X)=X\otb M$ is $ {}^*R_M:\Mo\to \ca$ given by ${}^*R_M(N)=\uhom(M,N).$ In particular, for any $X\in \ca$, the right adjoint to $R_X$ is $R_{X^*}$.

The next Lemma, although technical, will be crucial later when we relate two different notions of adjoint algebras.

\begin{lema}\label{duals-in-modcat} Let $\Mo$ be an exact indecomposable left $\ca$-module category. Let $M\in \Mo$, $X\in \ca$, and let $$\delta:{}^* (R_M\circ R_X)\to  R_{X^*}\circ {}^*R_M$$ be the natural isomorphim defined in Lemma \ref{dual-of-composition}. Then, for any $N\in \Mo$
$$ \delta_N=\mathfrak{b}^1_{X,M,N}.$$
Recall from Section \ref{subsection:internal hom} the definition of $\mathfrak{b}^1_{X,M,N}$.
\end{lema}
\pf We will assume that $\Mo$ is strict. We will make use of the notation of Section \ref{subsection:internal hom}, included the isomorphisms $\psi, \phi$. Observe that $R_M\circ R_X = R_{X\otb M}. $ From the definition of $\phi$ given in \eqref{iso-dualsl}, we have that $\delta$ is equal to
$$ (\id_{R_{X^*} \circ {}^*R_M}  \circ \ev_{R_{X\otb M}} )( \id_{R_{X^*}} \circ \coev_{R_M}\circ   \id_{R_X\circ {}^*R_{X\otb M} })( \coev_{R_X}\circ  \id_{{}^*R_{X\otb M}}).
$$
If $N\in \Mo$, then, using that
$$\big( \coev_{R_X}\big)_Y=\id_Y\ot \coev_X, \,\,\big(  \coev_{R_M}\big)_X=\coev^{\Mo}_{X,M},$$
$$ \big( \ev_{R_M}\big)_N=\ev^{\Mo}_{M,N},$$
 we obtain 
\begin{align*} (\id_{R_{X^*} \circ {}^*R_M}  \circ \ev_{R_{X\otb M}} )_N&= \uhom(M, \ev^{\Mo}_{X\otb M,N})\ot \id_{X^*},\\
 ( \id_{R_{X^*}} \circ \coev_{R_M}\circ   \id_{R_X\circ {}^*R_{X\otb M} })_N&=\coev^{\Mo}_{\uhom(X\otb M, N)\ot X, M}\ot \id_{X^*},\\
 ( \coev_{R_X}\circ  \id_{{}^*R_{X\otb M}})_N&=\id_{\uhom(X\otb M, N)}\ot  \coev_X.
\end{align*}
Whence, for any  $Z\in \ca$, and any $\alpha\in \Hom_\ca(Z, \uhom(X\otb M, N))$ we have that 
\begin{align*} \delta_N\alpha= \big( \uhom(M, \ev^{\Mo}_{X\otb M,N})\ot \id_{X^*}\big)& \big(\coev^{\Mo}_{\uhom(X\otb M, N)\ot X, M}\ot \id_{X^*}\big)\\
& \big(\alpha\ot \id_{X\ot X^*} \big)(\id_Z\ot \coev_X).
\end{align*}
On the other hand, it follows from the definition of $ \mathfrak{b}^1_{X,M,N}:\uhom(X\otb M, N)\to \uhom(M,N)\ot X^* $,  that
\begin{equation}  \mathfrak{b}^1_{X,M,N} \alpha = \big( \psi^{Z\ot X}_{M,N} (\phi^Z_{X\otb M, N}(\alpha)) \ot \id_{X^*}\big) \big(\id_Z\ot \coev_X\big),
\end{equation}
for any $Z\in \ca$, and any $\alpha\in \Hom_\ca(Z, \uhom(X\otb M, N))$. Thus $$\delta_N= \mathfrak{b}^1_{X,M,N}$$ if and only if $$\delta_N\alpha=\mathfrak{b}^1_{X,M,N} \alpha  $$ 
if and only if
\begin{align}\label{finish}\begin{split}  \uhom(M, \ev^{\Mo}_{X\otb M,N}) \coev^{\Mo}_{\uhom(X\otb M, N)\ot X, M}&
 \big(\alpha\ot \id_{X} \big)=\\
 &= \psi^{Z\ot X}_{M,N} (\phi^Z_{X\otb M, N}(\alpha))
 \end{split}
\end{align}
for any $\alpha\in \Hom_\ca(Z, \uhom(X\otb M, N))$. Using the   naturality of $\psi$, see \eqref{nat-psi-oneside}, we get that
\begin{align}\label{tech-ph-ps}\begin{split}\coev^{\Mo}_{\uhom(X\otb M, N)\ot X, M}
 (\alpha\ot \id_{X} )&=\psi^{\uhom(X\otb M, N)\ot X}_{M, \uhom(X\otb M, N)\ot X \otb M}(\id)  (\alpha\ot \id_{X} )\\
 &= \psi^{Z\ot X}_{M,\uhom(X\otb M, N)\ot X\otb M }(\alpha\ot \id_{X\otb M}).\end{split}
\end{align}
The first equality follows from the definition of $ \coev^{\Mo}_{X,M}$. Using the naturality of $\phi$, see \eqref{nat-phi-oneside}, we get that
\begin{align*} & \phi^{Z\ot X}_{M,N} \big(\uhom(M, \ev^{\Mo}_{X\otb M,N})\coev^{\Mo}_{\uhom(X\otb M, N)\ot X, M}
 (\alpha\ot \id_{X} )\big) =\\
 &= \ev^{\Mo}_{X\otb M,N}  \phi^{Z\ot X}_{M,\uhom(X\otb M, N)\ot X\otb M} (\coev^{\Mo}_{\uhom(X\otb M, N)\ot X, M}
 (\alpha\ot \id_{X} ))\\
 &=\ev^{\Mo}_{X\otb M,N}  (\alpha\ot \id_{X\otb M})\\
 &= \phi^{\uhom(X\otb M, N)}_{X\otb M, N} (\id)(\alpha\ot \id_{X\otb M})\\
 &=\phi^Z_{X\otb M, N}(\alpha).
\end{align*}
The second equality follows from \eqref{tech-ph-ps}, the third equality follows from the definition of $\ev^{\Mo}$, and the last one follows from \eqref{nat-phi-oneside2}. This equality implies \eqref{finish}, and the proof is finished.
\epf

Assume that $\D$ is another finite tensor category, and let $\No$ be an invertible $(\ca,\D)$-bimodule category. Define  $\theta^{\No}:\camod_e \to {}_\Do\text{Mod}_e$,  the pseudofunctor described as follows. Let $\Mo, \Mo'$ be  $\ca$-module categories. Then $\theta^{\No}(\Mo)=\Fun_\ca(\No,\Mo)$, and if $F\in \Fun_\ca(\Mo,\Mo')$, then 
$$\theta^{\No}(F):\Fun_\ca(\No,\Mo)\to \Fun_\ca(\No,\Mo'),$$
$$  \theta^{\No}(F)(H)=F\circ H,$$
for any $H\in \text{Fun}_\ca(\No,\Mo)$.

The next result seems to be part of the folklore of the area.  Since we were unable to find a reference in the literature we include a brief proof.
\begin{prop}\label{higher-Morita} The 2-functor $  \theta^{\No}:\camod_e \to {}_\Do Mod_e$ is a biequivalence of 2-categories.
\end{prop}
\pf Using \cite[Prop. 4.2]{ENO} we can assume that $\Do=\End_\ca(\No)^{rev}.$ The opposite category $\No^{\op}$ is a $(\Do,\ca)$-bimodule category. Let us denote by $\theta^{\No^{\op}}:{}_\Do Mod_e\to \camod_e$ the 2-functor given by $\theta^{\No^{\op}}(\Mo)=\Fun_\Do(\No^{\op},\Mo)$, for any $\Do$-module category $\Mo$. Let us prove that $\theta^{\No}\circ \theta^{\No^{\op}}\simeq \Id.$ Define the pseudonatural transformation
$$(\chi, \chi^0): \theta^{\No}\circ \theta^{\No^{\op}}\to\Id$$
as follows. For any $\Do$-module category $\Mo$ the functor $$\chi^0_\Mo:\Fun_{\ca}(\No, \Fun_{\Do}(\No^{\op}, \Mo)) \to \Mo,$$
$$\chi^0_\Mo(H)= \oint_{N\in \No} H(N)(N).$$
Here $\oint$ stands for the "\emph{relative}" end; a categorical tool developed in \cite{BM}. See \cite[Section 4.4]{BM} for details on this functor. 

If $\Mo, \Mo'$ are $\Do$-module categories, and $F: \Mo\to \Mo'$ is a $\Do$-module functor, we must define a module natural  transformations
$$\chi_F: F\circ \chi^0_\Mo\to \chi^0_{\Mo'}\circ   \theta^{\No} \theta^{\No^{\op}}(F).$$
If $H\in \Fun_{\ca}(\No, \Fun_{\Do}(\No^{\op}, \Mo))$,  the functor 
$$ \chi^0_{\Mo'}\circ   \theta^{\No} \theta^{\No^{\op}}(F):\Fun_{\ca}(\No, \Fun_{\Do}(\No^{\op}, \Mo))\to \Mo'  $$
evaluated in $H$ is
\begin{align*} \chi^0_{\Mo'}\circ   \theta^{\No} \theta^{\No^{\op}}(F)(H)&= \chi^0_{\Mo'}( \theta^{\No^{\op}}(F)\circ H)\\
&=\oint_{N\in \No} (\theta^{\No^{\op}}(F)\circ H)(N)(N)\\
&= \oint_{N\in \No} (F \circ H(N))(N)\\
&= F\big(  \oint_{N\in \No} H(N)(N) \big)=
F\circ \chi^0_\Mo(H).
\end{align*}
The fourth equality follows since the relative end is invariant under module functors, see \cite[Prop. 3.3 (iii)]{BM}. Hence, we can define $\chi_F$ to be the identity natural transformation. Since $\chi^0_\Mo$ is an equivalence of categories, see \cite[Thm. 4.12 ]{BM}, one can show that the pseudonatural transformation $(\chi,\chi^0)$ is an equivalence. Whence, $\theta^{\No}\circ \theta^{\No^{\op}}\simeq \Id.$ The proof that  $\theta^{\No^{\op}}\circ \theta^{\No}\simeq \Id$ follows in a completely similar way.
\epf

\subsection{The center of a 2-category} 

If $\Bc$ is a 2-category, the center of $\Bc$ is the strict monoidal category $Z(\Bc)=\operatorname{PseuNat}(\Id_\Bc,\Id_\Bc)$, consisting of pseudonatural transformations of the identity pseudofunctor $\Id_\Bc$, see \cite{MS}.  Explicitly, objects in    $Z(\Bc)$ are pairs $(V,\sigma)$, where
 $$V=\{V_A\in \Bc(A,A) \text{ 1-cells},  A\in \Bc \}, $$   
 $$\sigma=\{\sigma_X: V_B\circ X\to X\circ V_A\},$$
 where, for any $X\in \Bc(A,B)$, $\sigma_X$ is a natural isomorphism 2-cell such that
\begin{equation}\label{braid-cent1} \sigma_{I_A}=\id_{V_A},
\sigma_{X\circ Y}= (\id_{X}\circ \sigma_Y)(\sigma_X\circ \id_{Y}),
\end{equation}
 for any 0-cells $A,B, C\in \Bc$, and any pair of 1-cells
 $X\in \Bc(A,B)$, $Y\in \Bc(C,B)$.

If $(V,\sigma)$, $(W,\tau)$ are two objects in $Z(\Bc)$, a morphism $f:(V,\sigma)\to (W,\tau)$ in $Z(\Bc)$ is a collection of 2-cells $f_A:V_A\Rightarrow W_A$, $A\in\Bc$ such that
\begin{equation}\label{monoidal-cent0} 
(\id_{X}\circ f_A)\sigma_X=\tau_X (f_B\circ \id_{X}),
\end{equation}
for any 1-cell $X\in\Bc(A,B)$. The category $Z(\Bc)$ has a monoidal product defined as follows. Let $(V,\sigma), (W,\tau)\in Z(\Bc)$ be two objects, then $(V,\sigma)\ot (W,\tau)=(V\ot W, \sigma\ot \tau)$, where 
\begin{equation}\label{monoidal-cent1} (V\ot W)_A=V_A\circ W_A,\quad
(\sigma\ot \tau)_X = (\sigma_X\circ \id_{W_A})(\id_{V_B}\circ \tau_X),
\end{equation}
for any 0-cells 
$A, B\in \Bc$, and $X\in \Bc(A,B)$. The monoidal product at the level of morphisms is the following.
If $(V,\sigma), (V',\sigma'), (W,\tau), (W',\tau')\in Z(\Bc)$ are objects, and $f:(V,\sigma)\to (V',\sigma'),$ $f':(W,\tau)\to (W',\tau')$ are morphisms in $Z(\Bc)$, then $f\ot f':(V,\sigma)\ot (V',\sigma') \to (W,\tau)\ot (W',\tau')$ is defined by
$$(f\ot f')_A=f_A\circ f'_A, $$
for any 0-cell $A$. The unit $(\uno,\iota)\in Z(\Bc)$ is the object
$$ \uno_A=I_A, \quad \iota_X=\id_X,$$
for any 0-cells $A, B$ and any 1-cell $X\in \Bc(A,B).$

\begin{rmk} If $\ca$ is a finite tensor category, the center of the 2-category $\underline{\ca}$, described in Example \ref{exa2-cat-monoidal}, coincides with the Drinfeld center $Z(\ca)$ of $\ca$.
\end{rmk}
\begin{prop}\label{center-2cat-monoid} Let $\Bc, \widetilde{\Bc}$ be finite 2-categories. Any biequivalence $\Fc:\Bc\to \wbc$ induces  a monoidal equivalence $\widehat{\Fc}:Z(\Bc)\to Z(\wbc).$

\end{prop}
\pf Assume that  $(\Fc,\alpha):\Bc\to \wbc$, $(\Gc,\alpha'):\wbc\to \Bc$ is a pair of biequivalences, that is, there is a pseudonatural equivalence $\chi:\Fc\circ \Gc\to \Id_{\wbc}$. Let $\tau:\Id_{\wbc}\to \Fc\circ \Gc$ be the inverse of $\chi$. This means that $\chi\circ \tau\sim \id_{\Id}$, and $\tau\circ \chi\sim \id_{\Fc\circ \Gc}$.
In particular, for any pair of 0-cells $C, D\in \wbc^0,$
$\chi^0_C\in \wbc(\Fc(\Gc(C),C)$, $\tau^0_C\in \wbc(C, \Fc(\Gc(C))$ are 1-cells, and for any 1-cell $Y\in \wbc(C,D)$ we have 2-cells
$$\chi_Y: Y\circ \chi^0_C\Longrightarrow \chi^0_D\circ \Fc(\Gc(Y)),$$
$$\tau_Y: \Fc(\Gc(Y))\circ \tau^0_C\Longrightarrow\tau^0_D\circ Y. $$
Let $\omega: \tau\circ \chi\to \id_{\Fc\circ \Gc}$ be an invertible modification.

Let be $(V,\sigma)\in Z(\Bc) $. Let us define $\widehat{\Fc}(V,\sigma)\in Z(\wbc)$ as follows. For  any 0-cell $C\in \wbc^0$
$$ \widehat{\Fc}(V)_C=\chi^0_C\circ \Fc(V_{\Gc(C)})\circ \tau^0_C.  $$
If $C, D\in \wbc^0$ are 0-cells and $Y\in \wbc(C,D)$, then
$$\widehat{\Fc}(\sigma)_Y:\widehat{\Fc}(V)_D\circ Y\to Y\circ\widehat{\Fc}(V)_C,  $$
is defined to be the composition
$$\chi^0_D\circ \Fc(V_{\Gc(D)})\circ \tau^0_D \circ Y\xrightarrow{\id\circ (\tau_Y)^{-1}}\chi^0_D\circ \Fc(V_{\Gc(D)})\circ  \Fc(\Gc(Y))\circ \tau^0_C \to$$
$$\xrightarrow{\id\circ \alpha \circ \id}\chi^0_D\circ \Fc( V_{\Gc(D)}\circ  \Gc(Y)) \circ \tau^0_C \xrightarrow{\id\circ \Fc(\sigma_{\Gc(Y)})\circ \id }\chi^0_D\circ \Fc( \Gc (Y) \circ V_{\Gc(C)}) \circ \tau^0_C$$
$$\xrightarrow{\id\circ \alpha^{-1} \circ \id}  \chi^0_D\circ \Fc( \Gc(Y)) \circ \Fc(V_{\Gc(C)} ) \circ \tau^0_C\xrightarrow{ (\chi_Y)^{-1}\circ \id } Y\circ \chi^0_C\circ  \Fc(V_{\Gc(C)} ) \circ \tau^0_C.$$

Here we are omitting the subscripts of $\alpha$ as a space saving measure. It follows from the property of the half-braiding $\sigma $ \eqref{braid-cent1} and from \eqref{pseudonat-def1} that $\widehat{\Fc}(\sigma)$ satisfies \eqref{braid-cent1}. Thus $\widehat{\Fc}$ defines a functor, and one can prove that $\widehat{\Gc}$ is a quasi-inverse of $\widehat{\Fc}$.

Let us define a monoidal structure of the functor $\widehat{\Fc}$. Let be $(V,\sigma), (W,\gamma)\in Z(\Bc) $, then define
$$ \xi_{(V,\sigma), (W,\gamma)}:\widehat{\Fc}(V,\sigma)\ot  \widehat{\Fc}(W,\gamma)\to \widehat{\Fc}((V,\sigma)\ot (W,\gamma))$$
as follows. For any 0-cell $C\in\wbc^0$ 
$$\xi_C: \chi^0_C\circ \Fc(V_{\Gc(C)})\circ \tau^0_C\circ\chi^0_C\circ \Fc(W_{\Gc(C)})\circ \tau^0_C\to \chi^0_C\circ  \Fc(V_{\Gc(C)}\circ W_{\Gc(C)})\circ \tau^0_C,$$
$$\xi_C=  \big(\id\circ \alpha_{V_{\Gc(C)},W_{\Gc(C)} }\circ\id\big)\big(\id\circ \omega_C\circ \id\big). $$

Here $\xi_C=\big( \xi_{(V,\sigma), (W,\gamma)}\big)_C$. The proof that $(\widehat{\Fc}, \xi)$ is monoidal follows from the fact that $\omega$ is a modification and $\alpha$ satisfies \eqref{pseudofunc-axiom1}. 
\epf

\begin{teo}\label{equivalence-centers} Let  $\ca$ be a finite tensor category. There are monoidal equivalences  $Z(\camod)\simeq Z(\camod_e) \simeq Z(\ca)^{\rev}.$ 
\end{teo}
\pf We define a pair of  functors $\Phi: Z(\camod)\to Z(\ca)^{rev}$, $\Psi: Z(\ca)^{rev}\to Z(\camod)$, that will establish an equivalence. First, let us describe what an object in  $Z(\camod)$ looks like. If $((V_\Mo),\sigma)\in Z(\camod)$, then for any left $\ca$-module category $\Mo$, $(V_\Mo, c^{V_\Mo}):\Mo\to \Mo$ is a $\ca$-module functor, and for any $\ca$-module functor  $(G,d):\Mo\to \No$, $\sigma_G:V_\No\circ G\to  G\circ V_\Mo$ is a family of natural module isomorphisms that satisfies \eqref{braid-cent1}. For later use, let us recall that  the $\ca$-module structure of the functors $V_\No \circ G$ and $G \circ V_\Mo$ are given by $c^{V_\No}_{X,G(M)} V_\No (d_{X,M})$ and $d_{X,M} G(c^{V_\Mo}_{X,M})$ respectively, for any $X \in \ca$ and $M \in \Mo$.

Let us pick $((V_\Mo),\sigma)\in Z(\camod)$, and define $V:=V_\ca(\uno)\in \ca$. Here $\ca$ is considered as a left $\ca$-module via the regular action. Let us prove that  $V$ belongs to the center $Z(\ca),$ that is, it posses a half-braiding. 

For any $X\in \ca$, the functor $R_X:\ca\to \ca$, $R_X(Y)=Y\ot X$ is a left $\ca$-module functor. Thus we can consider the isomorphism $(\sigma_{R_X})_\uno:V_\ca(X)\to  V\ot X$. Since $V_\ca$ is a module functor, it comes with natural isomorphisms 
$$ c^{V_\ca}_{X,Y}: V_\ca(X\ot Y)\to X\ot V_\ca(Y),$$
for any $X, Y\in \ca$. In particular, we have natural isomorphisms
$$  c^{V_\ca}_{X,\uno}:  V_\ca(X) \to X\ot V,$$
for any $X\in \ca$. For any $X\in \ca$ define 
\begin{equation}\label{half-brading-phi}
    \alpha^{\sigma}_X: V\ot X\to X\ot V, \quad \alpha^{\sigma}_X=  c^{V_\ca}_{X,\uno}(\sigma_{R_X})^{-1}_\uno. 
\end{equation}
It follows by a straightforward computation that $\alpha^\sigma$ is a half-brading for $V$, that is $(V,\alpha^\sigma)\in Z(\ca). $ Therefore, we define $\Phi((V_\Mo), \sigma)=(V,\alpha^\sigma)$.

Now, given an object $(V,\alpha)\in Z(\ca)$, we define the functor $\Psi: Z(\ca)^{rev}\to Z(\camod)$ as $\Psi(V,\alpha)=((V_\Mo),\sigma^\alpha)$, where for any $\ca$-module category $\Mo$ 
$$V_\Mo:\Mo\to \Mo, \quad V_\Mo(M)=V\otb M,$$
for any $M\in \Mo$. The module structure of $V_\Mo$ is given by
$$c^{V_\Mo}_{X,M}:V_\Mo(X\otb  M)=V\otb (X\otb M)\to X\otb (V\otb M), $$  $$c^{V_\Mo}_{X,M}=m_{X,V,M} (\alpha_X\ot\id_M) m^{-1}_{V,X,M},$$
for any $X\in \ca$, $M\in \Mo$. Here $m$ denotes the associativity constraint of the module category $\Mo$.

If $(F,d): \Mo \to \No$ is a left $\ca$-module functor, $$\sigma^{\alpha}_F: V_\No \circ F \to F \circ V_\Mo, \quad \sigma^{\alpha}_F= d^{-1}_{V,-}.$$ 

For any 
$$f:((V_\Mo), \sigma) \to ((V_\ca (\uno))_\Mo , \sigma^{\alpha^\sigma}).$$ For each $\ca$-module $\Mo$, we must define a natural transformation $f_\Mo : V_\Mo \rightarrow (V_\ca (\uno))_\Mo$.
For any $M\in \Mo$, the functor $R_M:\ca\to \Mo$, $R_M(X)=X\otb M$ is a $\ca$-module functor. Thus, we can consider the half-brading
$$\sigma_{R_M}:V_{\Mo}\circ R_M\to R_M\circ V_{\ca}. $$
In particular we have the map
$$ (\sigma_{R_M})_{\uno}:V_{\Mo}(M) \to V_\ca (\uno) \otb M.$$
Hence, define $(f_\Mo)_M = (\sigma_{R_M})_\uno$ for any  $M\in \Mo$. Observe that $f_\Mo$ is a natural $\ca$-module isomorphism. Let us prove that $f$ defines a morphism in $Z(\camod)$. We need to show that $f$ satisfies equation \eqref{monoidal-cent0}, that is 
 
 \begin{equation}\label{monoidal-cent0-f}
     (\id_F \circ f_\Mo)_M(\sigma_F)_M = (\sigma^{\alpha^\sigma}_F)_M (f_\No \circ \id_F)_M
 \end{equation}
 for any $\ca$-module functor $(F,d): \Mo \rightarrow \No$ and any $M \in \Mo$. The left hand side of \eqref{monoidal-cent0-f} is
 \begin{align*}
    (\id_F \circ f_\Mo)_M(\sigma_F)_M & = F ((f_\Mo)_M) (\sigma_F)_M\\
    & = (\sigma_{F\circ R_M})_\uno \\ 
    & = d^{-1}_{V_\ca(\uno),M} (\sigma_{R_{F(M)}})_\uno \\ 
    & =(\sigma^{\alpha^\sigma}_F)_M (f_\Mo \circ \id_F)_M.
 \end{align*}
 
 where the first equality is the composition of natural transformations and the second equality follows from \eqref{braid-cent1}. The third equality follows from the fact that $\sigma$ is a natural module transformation. That is, since $d_{-,M} : F \circ R_M \rightarrow R_{F(M)}$ is the mdule structure of $F$, we have $(d_{-,M} \circ \id_{V_\ca}) \sigma_{F\circ R_M} = \sigma_{R_{F(M)}} (\id_{V_\No} \circ d_{-,M})$ and therefore $d_{V_\ca(\uno),M} (\sigma_{F\circ R_M})_\uno = (\sigma_{R_{F(M)}})_\uno$. The last equality is simply the definiton of $\sigma^{\alpha^\sigma}_F$.

This proves that $\Psi \Phi \simeq \Id_{Z(\camod)}$ . The proof of $\Phi \Psi \simeq \Id_{Z(\ca)}$ is straightfordward.

Let us prove now that the functor $\Phi$ is monoidal. Let us take two objects $((V_\Mo), \sigma), ((W_\Mo), \gamma) \in Z(\camod)$. Then
$$\Phi((V_\Mo), \sigma) \otimes^{\rev} \Phi((W_\Mo), \gamma) = (W_\C(\uno) \otimes V_\C(\uno), \alpha),$$
where, according to \eqref{monoidal-cent1}, the half-brading of the tensor product of two objects is 
$$\alpha_X= (\alpha^{\gamma}_X \otimes \id_{V_\C(\uno)}) (\id_{W_\C(\uno)} \otimes \alpha^{\sigma}_X)$$ for all $X \in \C$. On the other hand
$$\Phi(((V_\Mo), \sigma) \otimes ((W_\Mo), \gamma))= ( (V \otimes W)_\C(\uno), \alpha^{\sigma \otimes \gamma} )$$
where, according to \eqref{half-brading-phi}, the half-braiding is
$$\alpha^{\sigma \otimes \gamma}_X = c^{V_\Mo \circ W_\Mo}_{X, \uno} ( (\sigma \otimes \gamma)_{R_X})^{-1}_\uno.$$
$$c^{V_\Mo \circ W_\Mo}_{X, Y} = c^{V_\C}_{X,Y} V_\C (c^{W_\C}_{X,Y}) \quad \text{for all} \quad X,Y \in \C,$$
and 
$$((\sigma \otimes \gamma)_F)_M = (\sigma_F)_{W_\Mo(M)} V_\No((\gamma_F)_M),$$ 
for any $\C$-module functor $F: \Mo \rightarrow \No$  and any $M \in \Mo$.

The monoidal structure of $\Phi$ is defined as follows. 
$$\zeta^{\Phi}_{((V_\Mo), \sigma), ((W_\Mo), \gamma)} : \Phi((V_\Mo), \sigma) \otimes^{\rev} \Phi((W_\Mo), \gamma) \rightarrow \Phi(((V_\Mo), \sigma) \otimes ((W_\Mo), \gamma)),$$ 
$$\zeta^\Phi_{((V_\Mo), \sigma), ((W_\Mo), \gamma)}=(c^{V_\ca}_{W_\ca(\uno),\uno})^{-1}.$$

Let us show that $\zeta^\Phi$ is a morphism in $Z(\C)$. For this, we need to prove that it fulfills \eqref{0braid-cent1}, that is
\begin{equation}\label{monoidal-Phi} ((c^{V_\ca}_{W_\ca(\uno),\uno})^{-1} \otimes \id_X) \alpha_X = \alpha^{\sigma \otimes \gamma}_X (\id_X \otimes (c^{V_\ca}_{W_\ca(\uno),\uno})^{-1}),
\end{equation}
for any $X \in \C$. The left hand side of \eqref{monoidal-Phi} is

\begin{align*} & ((c^{V_\ca}_{W_\ca(\uno),\uno})^{-1} \otimes \id_X) \alpha_X = (\id_X \otimes (c^{V_\ca}_{W_\ca(\uno),\uno})^{-1}) (\alpha^{\gamma}_X \otimes \id_{V_\ca(\uno)} )\\ & ( \id_{W_\ca(\uno)} \otimes \alpha^{\sigma}_X )\\
&=c^{V_\ca}_{X,W_\ca(\uno)} (c^{V_\ca}_{X \otimes W_\ca(\uno),\uno})^{-1} (c^{W_\ca}_{X, \uno} (\gamma_{R_X})^{-1}_{\uno} \otimes \id_{V_\ca(\uno)})(\id_{W_\ca(\uno)} \otimes c^{V_\ca}_{X,\uno})\\ & \qquad (\id_{W_\ca(\uno)} \otimes (\sigma_{R_X})^{-1}_\uno)\\
&=c^{V_\ca}_{X,W_\ca(\uno)} (c^{V_\ca}_{X \otimes W_\ca(\uno),\uno})^{-1} (c^{W_\ca}_{X, \uno} (\gamma_{R_X})^{-1}_{\uno} \otimes \id_{V_\ca(\uno)}) (\id_{W_\ca(\uno)} \otimes c^{V_\ca}_{X,\uno}) c^{V_\ca}_{W_\ca(\uno),X} \\ & \qquad (\sigma_{R_X})^{-1}_{W_\ca(\uno)} ((c^{V_\ca}_{W_\ca(\uno),\uno})^{-1} \otimes \id_X)\\
&=c^{V_\ca}_{X,W_\ca(\uno)}     (c^{V_\ca}_{X \otimes W_\ca(\uno), \uno})^{-1} (c^{W_\ca}_{X,\uno} \otimes \id_{V_\ca(\uno)}) ((\gamma_{R_X})^{-1}_\uno \otimes \id_{V_\ca(\uno)})c^{V_\ca}_{W_\ca(\uno) \otimes X, \uno}                        \\ & \qquad    (\sigma_{R_X})^{-1}_{W_\ca(\uno)} ((c^{V_\ca}_{W_\ca(\uno),\uno})^{-1} \otimes \id_X)\\
&= c^{V_\ca}_{X,W_\ca(\uno)}       (c^{V_\ca}_{X \otimes W_\ca(\uno),\uno})^{-1} (c^{W_\ca}_{X,\uno} \otimes \id_{V_\ca(\uno)}) c^{V_\ca}_{W_\ca(X),\uno} V_\ca((\gamma_{R_X})^{-1}_\uno )       (\sigma_{R_X})^{-1}_{W_\ca(\uno)} \\ & \qquad ((c^{V_\ca}_{W_\ca(\uno),\uno})^{-1} \otimes \id_X)\\
&= c^{V_\ca}_{X,W_\ca(\uno)}        V_\ca(c^{W_\ca}_{X,\uno}) V_\ca((\gamma_{R_X})^{-1}_\uno )             (\sigma_{R_X})^{-1}_{W_\ca(\uno)} ((c^{V_\ca}_{W_\ca(\uno),\uno})^{-1} \otimes \id_X)\\
&= \alpha^{\sigma \otimes \gamma}_X (\id_X \otimes (c^{V_\ca}_{W_\ca(\uno),\uno})^{-1}).
\end{align*}

The first equality follows from the definition of $\alpha_X$, the second and fourth equalities follow from the axioms of $c^{V_\ca}$. The third equality follows  from the fact that for a given $Y \in \ca$ and $R_Y : \ca \rightarrow \ca$, $\sigma_{R_Y}$ is a natural module isomorphism satisfiying \eqref{modfunctor3}, and the module structures of $V_\ca \circ R_Y$ and $R_Y \circ V_\ca$ are given by $c^{V_\ca}_{X, Z \otimes Y}$ and $R_Y (c^{V_\ca}_{X,Z})$ for any $X,Z \in \ca$. Then from \eqref{modfunctor3} we have $$ (\id_X \otimes (\sigma_{R_Y})^{-1}_\uno)(c^{V_\ca}_{X,\uno} \otimes \id_Y)= c^{V_\ca}_{X,Y} (\sigma_{R_Y})^{-1}_X,$$ for all $X,Y \in \ca$.
The fifth and sixth equalities follow from the naturality of $c^{V_\ca}_{-, \uno}$, and the seventh follows from the definition of $\alpha^{\sigma \otimes \gamma}_X $.

It is straightforward that $\zeta^\Phi$ satisfies
the axiom required for $(\Phi,\zeta^\Phi)$ to be a monoidal functor.

The proof of the equivalence $Z(\camod_e) \simeq Z(\ca)^{\rev}$ follows \textit{mutatis mutandis}.
\epf

We want to apply Proposition \ref{center-2cat-monoid} to the 2-category of $\ca$-modules. Let $\ca$ be a finite tensor category and $\No$ be an indecomposable exact left $\ca$-module category. Set $\Do=(\ca^*_{\No})^{\rev}$. Then $\No$ is an invertible exact $(\ca, \Do)$-bimodule category. Thus, we can consider the 2-equivalence $\theta^\No: \camod_e\to {}_\Do\text{Mod}_e$ presented in Proposition \ref{higher-Morita}. According to Proposition \ref{center-2cat-monoid}, this 2-equivalence induces a monoidal equivalence 
$\widehat{\theta}^\No:Z(\camod_e)\to Z({}_\Do\text{Mod}_e)$.

There is a commutative diagram of monoidal equivalences

\begin{equation}\label{theta-equival-center}
\xymatrix{
Z(\ca)^{\rev}\ar[d]_{\simeq }\ar[rr]^{\theta}&& Z(\Do)^{\rev}\ar[d]^{\simeq} \\
Z(\camod_e)\ar[rr]^{\widehat{\theta}^\No}&& Z({}_\Do\text{Mod}_e)}
\end{equation}
Equivalences in the vertical arrows come from Theorem \ref{equivalence-centers}, and the functor $\theta: Z(\ca)\to Z(\Do)$ is given by $\theta(V,\sigma):\No\to \No,$ $\theta(V,\sigma)(N)=V\otb N$, for all $N\in \No$. The functor $\theta$ coincides with the one presented by Shimizu in \cite[Theorem 3.13]{Sh2}. See also \cite{Sch}.

\section{The adjoint algebra for finite 2-categories}\label{Section:adj-2cat}

Throughout this section $\Bc$ will denote a finite 2-category. For any pair of 0-cells $A, B$ of $\Bc$ we define the 1-cell
\begin{equation}\label{coend-2cat} \ele(A,B)=\int_{X\in \Bc(A,B)} {}^* X\circ X\in \Bc(A,A).
\end{equation}
For any $X\in \Bc(A,B)$ we will denote by $\pi^{(A,B)}_X: \ele(A,B) \xrightarrow{..}  {}^* X\circ X$ the dinatural transformations associated to this end. Since  all categories $\Bc(A,B)$ are finite categories, the end $\ele(A,B)$ always exists.

\begin{prop}\label{half-braidig-2cat}  Assume that $A,B,C$ are 0-cells. There exists a natural isomorphism 
$$ \sigma^{B}_X:\ele(A,B)\circ X  \Rightarrow X\circ \ele(C,B)$$
such that the diagram
\begin{equation}\label{definition-diagram-braiding}
\xymatrix@C=70pt@R=16pt{
	\ele(A,B)\circ X
	\ar[r]^{\pi^{(A,B)}_Y \circ \id_X}
	\ar[d]_{\sigma^{B}_X}
	& {}^*Y\circ Y\circ X \\
	X\circ \ele(C,B)
	\ar[r]_{\id_X\circ \pi^{(C,B)}_{Y\circ X}}
	& X\circ {}^* X\circ {}^* Y\circ Y\circ X 	\ar[u]_{\ev_X\circ\id_{{}^*Y\circ Y\circ X}} .
}
\end{equation}
is commutative for any $X\in \Bc(C,A), Y\in \Bc(A,B)$. With this map the collection $\adj_B=((\ele(A,B)_{A\in \Bc^0 }), \sigma^{B})$ is an object in the center $Z(\Bc)$ for any 0-cell $B\in \Bc^0$. 
\end{prop}
\pf  The end $\int_{Y\in \Bc(A,B)} {}^* Y\circ Y\circ X$ is equal to $\ele(A,B)\circ X$ with dinatural transformations $\pi^{(A,B)}_Y \circ \id_X$. Since the maps
$$(\ev_X\circ\id_{{}^*Y\circ Y\circ X})(\id_X\circ \pi^{(C,B)}_{Y\circ X}):X\circ \ele(C,B) \to {}^* Y\circ Y\circ X$$ are dinatural transformations, it follows from the universal property of the end that,  for any $X\in \Bc(C,A)$, there exists a map $\overline{\sigma}^{B}_X: X\circ \ele(C,B) \to \ele(A,B)\circ X$ such that
\begin{equation}\label{prop-inv-sigma} (\pi^{(A,B)}_Y \circ \id_X)\overline{\sigma}^{B}_X=  ( \ev_X\circ\id_{{}^*Y\circ Y\circ X})(\id_X\circ \pi^{(C,B)}_{Y\circ X}).
\end{equation}
It follows easily that morphisms $\overline{\sigma}^{B}_X$ are natural in $X$. Let us prove that for any $Y\in \Bc(A,B)$
\begin{align}\label{braid-eq-inv}  \overline{\sigma}^{B}_{Y\circ X}= (\overline{\sigma}^{B}_Y\circ \id_X)(\id_Y\circ \overline{\sigma}^{B}_X).
\end{align}
Let $E$ be another 0-cell, and $Z\in \Bc(D,E)$ be a 1-cell. Then it follows from \eqref{prop-inv-sigma} that
\begin{align*}( \pi^{(D,B)}_{Z}\circ \id_{Y\circ X})  \overline{\sigma}^{B}_{Y\circ X}= ( \ev_{Y\circ X}\circ\id_{{}^*Z\circ Z\circ Y\circ X})(\id_{Y\circ X}\circ \pi^{(C,B)}_{Z\circ Y\circ X}).
\end{align*}
On the other hand
\begin{align*}&( \pi^{(D,B)}_{Z}\circ \id_{Y\circ X})(\overline{\sigma}^{B}_Y\circ \id)(\id\circ \overline{\sigma}^{B}_X)=(( \pi^{(D,B)}_{Z}\circ \id_{Y})\overline{\sigma}^{B}_Y\circ \id_X)(\id_Y\circ \overline{\sigma}^{B}_X)\\
&= ( \ev_Y\circ\id_{{}^*Z\circ Z\circ Y\circ X})(\id_Y\circ \pi^{(C,B)}_{Z\circ Y}\circ \id_X)(\id_Y\circ \overline{\sigma}^{B}_X)\\
&=( \ev_Y\circ\id_{{}^*Z\circ Z\circ Y\circ X}) ( \id_Y\circ \ev_X\circ\id_{{}^*(Z\circ Y)\circ Z\circ Y\circ X})( \id_{Y\circ X}\circ \pi^{(C,B)}_{Z\circ Y\circ X})
\end{align*}
Whence 
$$( \pi^{(D,B)}_{Z}\circ \id_{Y\circ X})  \overline{\sigma}^{B}_{Y\circ X}= ( \pi^{(D,B)}_{Z}\circ \id_{Y\circ X})(\overline{\sigma}^{B}_Y\circ \id)(\id\circ \overline{\sigma}^{B}_X).$$
Then, it follows from the universal property of the end, that  equation \eqref{braid-eq-inv} is satisfied. It remains to prove that for any $X$ the map $\overline{\sigma}^{B}_X$ is an isomorphism. The idea of the proof  of this fact is taken from \cite[Lemma 2.10]{DSS}.

For any $X\in \Bc(C,A)$ define 
$$\sigma^{B}_X=  (\id\circ ev_X)(\id_X\circ \overline{\sigma}^{B}_X\circ \id_X )(coev_X\circ \id).$$
One can prove, using the naturality of $ \overline{\sigma}^{B}_X$, the rigidity axioms and  \eqref{braid-eq-inv}, that $\sigma^{B}_X$ is indeed the inverse of $ \overline{\sigma}^{B}_X$.
\epf

\begin{rmk} Keep in mind that in diagram \eqref{definition-diagram-braiding} we are omitting the isomorphism ${}^*(Y\circ X)\simeq {}^*X\circ {}^*Y$.
\end{rmk}

The particular choice of the dinaturals in the coend \eqref{coend-2cat} does not change the equivalence class of the object $(\adj_B, \sigma^B)\in Z(\Bc)$. This is the next result.

\begin{lema}\label{invariance-dinaturalchoice}  Assume that for any 0-cell $A\in \Bc^0$,  $\gamma^{(A,B)}_X:  \ele(A,B) \xrightarrow{..}  {}^* X\circ X$ is another choice of dinatural transformations for this end. And let $$ \eta^{B}_X:\ele(A,B)\circ X  \Rightarrow X\circ \ele(C,B)$$ be the  half-braiding associated with these dinatural transformations. Then $((\ele(A,B)_{A\in \Bc^0 }), \sigma^{B}), ((\ele(A,B)_{A\in \Bc^0 }), \eta^{B})$ are isomorphic  as objects in the center $Z(\Bc)$. 
\end{lema}
\pf Since $\gamma^{(A,B)}$ are dinaturals, there exists a map $h_A:  \ele(A,B)\to  \ele(A,B)$ such that the diagram
\begin{equation}\label{diagram-invar}
\xymatrix{&\ele(A,B) 
\ar[dl]_{h_A}
\ar[dr]^{\pi^{(A,B)}_X}&\\
\ele(A,B) \ar[rr]^{\gamma^{(A,B)}_X}&&   {}^* X\circ X}
\end{equation}
commutes. Let $\widetilde{\adj}_B=((\ele(A,B)_{A\in \Bc^0 }), \eta^{B})$, and set $h: \adj_B\to \widetilde{\adj}_B$. Let us check that $h$  is a morphism in the center. We need to verify that, for any 1-cell $X\in \Bc(A,B)$
\begin{equation}\label{invariance-eq1}
(\eta^{B}_X)^{-1}(\id_X\circ h_A) = (h_A\circ \id_X)(\sigma^B_X)^{-1}.
\end{equation}
Let $C$ be antoher 0-cell, and $Y\in \Bc(B,C)$. To prove \eqref{invariance-eq1}, it is sufficient to prove that 
\begin{equation}\label{invariance-eq2}( \gamma^{(B,C)}_Y\circ \id_X)(\eta^{B}_X)^{-1}(\id_X\circ h_A)= ( \gamma^{(B,C)}_Y\circ \id_X)(h_A\circ \id_X)(\sigma^B_X)^{-1}.
\end{equation}
Using diagram \eqref{diagram-invar}, the right hand side of \eqref{invariance-eq2} is equal to
\begin{align*} &=( \pi^{(B,C)}_Y\circ \id_X)(\sigma^B_X)^{-1}\\
&=  (\ev_X\circ \id_{{}^*Y\circ Y\circ X}) (\id_X\circ \pi^{(A,C)}_{Y\circ X}).
\end{align*}
The second equality follows from diagram \eqref{definition-diagram-braiding}. On the other hand, the left hand side of \eqref{invariance-eq2} is equal to
\begin{align*} &= (\ev_X\circ \id_{{}^*Y\circ Y\circ X}) (\id_X\circ \gamma^{(A,C)}_{Y\circ X} h_A)\\
&= (\ev_X\circ \id_{{}^*Y\circ Y\circ X}) (\id_X\circ \pi^{(A,C)}_{Y\circ X}).
\end{align*}
The first equality follows from  diagram \eqref{definition-diagram-braiding}, and the second equality follows from \eqref{diagram-invar}.
\epf

In what follows, we will introduce a product for $\adj_B$. For any 0-cell $B\in \Bc^0$ define $m^B: \adj_B\ot \adj_B\to \adj_B, $ $u^B:\uno\to \adj_B$ as the unique morphisms in $Z(\Bc)$ such that 
\begin{align}\label{product-ch-2cat}
\begin{CD}
\ele(A,B)\circ \ele(A,B) @>\quad\quad m^B_A\quad\quad>> \ele(A,B)\\
@V\pi^{(A,B)}_X \circ \pi^{(A,B)}_X VV    @VV\pi^{(A,B)}_X V   \\
{}^*X\circ  X\circ {}^*X\circ X@>>\id_{{}^*X}\circ \ev_X \circ \id_{X}> {}^*X\circ X,
\end{CD}
\end{align}
\begin{equation}\label{product-ch-2cat-unit}
\xymatrix{I_A\ar[rr]^{u^B_A}
\ar[dr]^{coev_X}&&\ele(A,B)
\ar[dl]_{\pi^{(A,B)}_X}\\ & {}^*X \circ X,&}
\end{equation}
are commutative diagrams, for any 0-cell $A$ and any $X\in \Bc(A,B)$. It follows from the universal property of the end $\ele(A,B)$ that both maps $m^B_A, u^B_A$ exist.
\begin{prop} The object $\adj_B $ with product $m^B$ and unit $u^B$  is an algebra in the center $Z(\Bc)$.
\end{prop}
\pf We must show that
\begin{equation}\label{product11} m^B (u^B \ot \id)= \id=  m^B (id\ot u^B),
\end{equation}
\begin{equation}\label{product12}  m^B ( m^B\ot \id)= m^B (\id\ot  m^B).
\end{equation}
For any $A\in \Bc^0$ and any $X\in \Bc(A,B)$ we have that 
\begin{align*}  \pi^{(A,B)}_X m^B_A(u^B_A \circ \id)&= (\id_{{}^*X}\circ \ev_X \circ \id_{X}) (\pi^{(A,B)}_X \circ \pi^{(A,B)}_X) (u^B_A \circ \id)\\
&= (\id_{{}^*X}\circ \ev_X \circ \id_{X}) (coev_X\circ \pi^{(A,B)}_X)\\
&=\pi^{(A,B)}_X.
\end{align*}
The first equality follows from \eqref{product-ch-2cat}, the second one follows from \eqref{product-ch-2cat-unit}. The last equality is the rigidity axiom.  Hence \eqref{product11} follows from the universal property of the end. To prove \eqref{product12} it is enough to show that for any $A\in \Bc^0$ and any $X\in \Bc(A,B)$
$$  \pi^{(A,B)}_X m^B_A ( m^B_A\circ \id)= \pi^{(A,B)}_X m^B_A (\id\circ  m^B_A).$$
Using \eqref{product11}
\begin{align*}& \pi^{(A,B)}_X m^B_A ( m^B_A\circ \id)=  (\id_{{}^*X}\circ \ev_X \circ \id_{X}) (\pi^{(A,B)}_X m^B_A\circ \pi^{(A,B)}_X)\\
&=  (\id_{{}^*X}\circ \ev_X \circ \id_{X}) (\id_{{}^*X}\circ \ev_X \circ \id_{X}\circ\id_{{}^*X\circ X}) ( \pi^{(A,B)}_X\circ  \pi^{(A,B)}_X\circ  \pi^{(A,B)}_X).
\end{align*} 
On the other hand
\begin{align*}& \pi^{(A,B)}_X m^B_A (\id\circ  m^B_A)= (\id_{{}^*X}\circ \ev_X \circ \id_{X}) (\pi^{(A,B)}_X\circ \pi^{(A,B)}_X m^B_A)\\
&= (\id_{{}^*X}\circ \ev_X \circ \id_{X}) (\id_{{}^*X\circ X} \circ \id_{{}^*X}\circ \ev_X \circ \id_{X})  ( \pi^{(A,B)}_X\circ  \pi^{(A,B)}_X\circ  \pi^{(A,B)}_X).
\end{align*} 
Since both are equal, we get the result.
\epf

\begin{defi} For any finite 2-category $\Bc$ and any 0-cell $B$ of $\Bc$,  $\adj_B$ is the \textit{adjoint algebra} of $B$.
\end{defi}

\begin{lema}\label{iso-adj-0cell} Assume that $B, C\in \Bc^0$ are equivalent 0-cells. Then, the adjoint algebras $\adj_B, \adj_C$ are isomorphic as algebras in the center $Z(\Bc)$.
\end{lema}
\pf Since $B$ and $C$ are equivalent 0-cells, there exist  1- cells $X \in \Bc (B,C)$ and  $Y \in \Bc (C,B)$ with isomorphisms $\alpha : X \circ Y \rightarrow I_C$ and $\beta : Y \circ X \rightarrow I_B$. Using Lemma \ref{dual-of-composition} (iii) we get that $Y={}^*X$ with evaluation and coevaluation given by
$$\coev_X=\beta^{-1},\,\, \ev_X= \alpha (\id_X\circ \beta\circ \id_Y)(\id_{X\circ Y}\circ \alpha^{-1}). $$
Also $X={}^*Y$ with evaluation and coevaluation given by
$$\coev_Y=\alpha^{-1}, \,\, \ev_Y=\beta (\id_Y \circ \alpha \circ \id_X)(\id_{Y \circ X} \circ \beta^{-1}).$$
Using remark \ref{interchange}, one can easily verify  that
\begin{equation}\label{dual-beta}
    {}^*(\beta^{-1})\circ \id_{Y\circ X}=\id_Y\circ \alpha\circ \id_X,
\end{equation}
\begin{equation}\label{dual-alpha} \id_{X\circ Y}\circ \alpha= \id_X\circ \ev_Y\circ \id_Y.
\end{equation}
For any 0-cell $A$  let 
$$\pi^{(A,B)}_W : \ele (A,B) \xrightarrow{..} {}^*W \circ W, \quad W\in \Bc(A,B),$$
and 
$$\xi^{(A,C)}_Z : \ele (A,C) \xrightarrow{..} {}^*Z \circ Z, \quad Z\in \Bc(A,C)$$
be the associated dinatural transformations to $\ele (A,B)$ and  $\ele (A,C)$ respectively.
The dinaturality of $\pi$ implies, that for any $W\in \Bc(A,B)$
\begin{align}\label{dinat-beta1}\begin{split} (\id_{ {}^*W} \circ \beta^{-1} \circ \id_W) \pi^{(A,B)}_W &= (\id_{{}^*W} \circ {}^*(\beta^{-1}) \circ \id_{Y \circ X \circ W}) \pi^{(A,B)}_{Y \circ X \circ W}\\
&=(\id_{{}^*W\circ Y} \circ \alpha \circ \id_{X \circ W}) \pi^{(A,B)}_{Y \circ X \circ W}.
\end{split}
\end{align}
Here, we have used \eqref{dual-beta}.
For any $W \in \Bc (A,B)$  define 
$$\delta_W : \ele (A,C) \rightarrow {}^*W \circ W$$ $$\delta_W = (\id_{ {}^*W} \circ \beta \circ \id_W) \xi^{(A,C)}_{X \circ W}. $$
For any $Z \in \Bc (A,C)$  define 
$$\gamma_Z : \ele (A,B) \rightarrow {}^*Z \circ Z$$ $$
\gamma_Z = (\id_{{}^*Z} \circ \alpha \circ \id_Z) \pi^{(A,B)}_{Y \circ Z}.$$
It follows by a straightforward computation, that $\delta$ and $\gamma$ are dinatural transformations. 
By the universal property of the end, there exist maps
$$ g_A : \ele (A,C) \rightarrow \ele (A,B),$$
$$f_A : \ele (A,B) \rightarrow \ele (A,C), $$
such that for any $W \in \Bc (A,B)$, $Z \in \Bc (A,C)$
\begin{align}\label{definitio-f-g}\pi^{(A,B)}_W g_A = \delta_W, \quad
    \xi^{(A,C)}_Z f_A = \gamma_Z.
\end{align}
Let us show that $f_A$ is the inverse of $g_A$. For any $W \in \Bc (A,B)$ we have that
\begin{align*}\pi^{(A,B)}_W g_A f_A & = \delta_W f_A\\
& = (\id_{ {}^*W} \circ \beta \circ \id_W) \xi^{(A,C)}_{X \circ W} f_A \\
& =  (\id_{ {}^*W} \circ \beta \circ \id_W) \gamma_{X \circ W} \\
& =  (\id_{ {}^*W} \circ \beta \circ \id_W)(\id_{{}^*(X \circ W)} \circ \alpha \circ \id_{X \circ W}) \pi^{(A,B)}_{Y \circ X \circ W}\\
& =  \pi^{(A,B)}_W 
\end{align*}
In the last equation we have used \eqref{dinat-beta1}. This proves that $g_A f_A=\id$. Analogously, one can prove that $f_A g_A=\id.$
The collection $(f_A)$ defines an isomorphism
$$f:  \adj_B \rightarrow \adj_C$$
Let us show that $f$ is indeed a morphism in $Z (\Bc)$. For this, we must prove that for any $W \in \Bc (A,E)$ 
$$ (\id_W \circ f_A) \sigma^B_W = \sigma^C_W (f_E \circ \id_W).$$
To prove this, it will be enough to see that 
\begin{equation}\label{equation-tech1}
    (\xi^{(E,D)}_T \circ \id_W)(\sigma^C_W)^{-1} (\id_W \circ f_A) \sigma^B_W = (\xi^{(E,D)}_T f_E \circ \id_W)
\end{equation}
for any $T \in \Bc (E,D)$. Using the definition of $f$, the right hand side of \eqref{equation-tech1} is equal to
$\gamma^{(E,D)}_T\circ \id_W.$ Using the definition of $\sigma^C$, the left hand side of \eqref{equation-tech1} is equal to
\begin{align*} &(\ev_W\circ \id_{{}^*T\circ T\circ W})(\id_W\circ\xi^{(A,D)}_{T\circ W})(\id_W\circ f_A) \sigma^B_W\\
&=(\ev_W\circ \id_{{}^*T\circ T\circ W}) (\id_W\circ \gamma_{T\circ W})\sigma^B_W= \gamma_T\circ \id_W.
\end{align*}
The first equality follows from the definition of $f$, and the second one follows from the definition of $\sigma^B$. Let us prove now that $f:  \adj_B \rightarrow \adj_C$ is an algebra morphism. We need to show that
$$f_A m^B_A=m^C_A (f_A\circ f_A), $$
for any 0-cell $A.$ Here $m^B$ is the multiplication of $\adj_B$. For this, it is enough to prove that
\begin{equation}\label{f-is-alg} \xi^{(A,C)}_Z f_A m^B_A=\xi^{(A,C)}_Z  m^C_A (f_A\circ f_A),
\end{equation}
for any 1-cell $Z\in \Bc(A,C).$ Using the definition of $f$, the left hand side of \eqref{f-is-alg} is equal to
\begin{align*} &=\gamma_Z m^B_A=(\id_{{}^*Z}\circ \alpha\circ \id_Z)\pi^{(A,B)}_Z m^B_Z \\
&= (\id_{{}^*Z}\circ \alpha\circ \id_Z) (\id_{{}^*Z\circ X}\circ \ev_{Y\circ Z}\circ \id_{Y\circ Z})(\pi^{(A,B)}_{Y\circ Z}\circ \pi^{(A,B)}_{Y\circ Z}) \\
&= (\id_{{}^*Z}\circ \alpha (\id_{X}\circ \ev_Y\circ \id_Y )\circ \id_Z) (\id_{{}^*Z\circ X\circ Y} \circ \ev_Z\circ\id_{X\circ Y\circ Z})\\
&(\pi^{(A,B)}_{Y\circ Z}\circ\pi^{(A,B)}_{Y\circ Z})
\end{align*}
The second equality follows from the definition of $\gamma,$ and the third equality follows from the definition of the product $m^B$
\eqref{product-ch-2cat}, and the last one follows from the formula for $\ev_{Y\circ Z}$. The right hand side of \eqref{f-is-alg} is equal to
\begin{align*} &=(\id_{{}^*Z}\circ \ev_Z\circ \id_Z)(\xi_Z\circ  \xi_Z)(f_A\circ f_A)\\
&=(\id_{{}^*Z}\circ \ev_Z\circ \id_Z)(\id_{{}^*Z}\circ \alpha\circ \id_{Z\circ {}^*Z}\circ \alpha\circ \id_Z) (\pi^{(A,B)}_{Y\circ Z}\circ\pi^{(A,B)}_{Y\circ Z})\\
&=(\id_{{}^*Z}\circ \alpha\circ  \alpha\circ \id_Z) (\id_{{}^*Z\circ X\circ Y} \circ \ev_Z\circ\id_{X\circ Y\circ Z}) (\pi^{(A,B)}_{Y\circ Z}\circ\pi^{(A,B)}_{Y\circ Z})
\end{align*}
The first equality follows from the definition given in \eqref{product-ch-2cat} of $m^C$, and the second equality follows from the definition of $f_A$. Now, that both sides of \eqref{f-is-alg} are equal is a consequence of \eqref{dual-alpha}.
\epf

At this point, we have to verify that our definition of the adjoint algebra of the 2-category ${}_\ca\text{Mod}$ coincides with the definition presented by Shimizu in \cite{Sh2}. This is one of the main results of this work and it is stated in the next result. Recall from Section \ref{SubSection:character algebra} the definition, due to Shimizu,  of the character algebra $\cha_\Mo\in Z(\ca)$ associated to any exact $\ca$-module category.

\begin{teo}\label{adjoint2cat=adjointS} Let $\ca$ be a finite tensor category and $\Mo$ be an exact indecomposable left $\ca$-module category. Let  $\Phi: Z(\camod)\to Z(\ca)$ be the equivalence presented in Theorem \ref{equivalence-centers}. Then $\Phi(\adj_\Mo)\simeq \cha_\Mo$ as algebra objects in $Z(\ca)$.
\end{teo}
\pf Since $\Phi(\adj_\Mo)=\ele(\ca,\Mo)(\uno)$, we will construct an isomorphism $\phi:\ele(\ca,\Mo)(\uno)\to  \cha_\Mo$ and prove that it is an algebra morphism in $Z(\ca)$.

Using Lemma \ref{iso-adj-0cell}, since any module category is equivalent to a strict one, we can assume that $\Mo$ is strict.  Recall from Section 
\ref{subsection:internal hom} the $\ca$-module functor $R_M:\ca\to \Mo$, $R_M(X)=X\otb M$, $X\in\ca$, and its right adjoint $R^{ra}_M: \Mo\to \ca$  given by the internal hom $ {}^* R_M=R^{\ra}_M(N)=\uhom(M,N)$, $M,N\in \Mo$.

This induces an equivalence $R:\Mo\to \Fun_\ca(\ca, \Mo)$, $R(M)=R_M$ for any $M\in \Mo$. Its quasi-inverse functor is $H:\Fun_\ca(\ca, \Mo)\to \Mo$, $H(F)=F(\uno)$.  For any module functor $(F,p)\in\Fun_\ca(\ca, \Mo)$,  define the natural isomorphisms
$$ \alpha: R\circ H\to \Id,$$
$$( \alpha_F)_X=(p_{X,\uno})^{-1}, $$
for any $X\in \ca$. Since $\Mo$ is strict, for any $M\in \Mo$, the functor $R_M$ has module structure given by the identity. In particular  
\begin{equation}\label{about-alpha-iso} \alpha_{R_M}=\id.
\end{equation}

We will denote by $\pi^{(\No, \Mo)}_F: \int_{F\in \Fun_\ca(\No, \Mo)} {}^* F\circ F \to  {}^* F\circ F $ the dinatural transformations of the end $\ele(\No, \Mo)$. We also consider the dinatural transformations 
$$ \eta_M: \int_{M\in \Mo}  {}^* R_M \circ R_M \to  {}^* R_M \circ R_M.$$

Using Proposition \ref{properties-end} (ii), there exists an isomorphism $$h:\int_{F\in \Fun_\ca(\No, \Mo)} {}^* F\circ F  \to  \int_{M\in \Mo}  {}^* R_M \circ R_M$$
such that the following diagram commutes
\begin{equation}\label{diagram-h}
\xymatrix@C=60pt@R=16pt{
	\int_{F\in \Fun_\ca(\ca, \Mo)} {}^* F\circ F 
	\ar[r]^{h}
	\ar[d]_{\pi^{(\ca, \Mo)}_F}
	& \int_{M\in \Mo}  {}^* R_M \circ R_M \ar[d]^{\eta_{H(F)}}\\
	{}^* F\circ F
	& {}^*R_{F(\uno)}\circ R_{F(\uno)} \ar[l]^{{}^* (( \alpha_F)^{-1})\circ \alpha_F}.
}
\end{equation}
Taking $F=R_M$, for any $M\in \Mo$, and using \eqref{diagram-h}, \eqref{about-alpha-iso} we get that
\begin{equation}\label{diagram-h-caseM}
\eta_Mh= \pi^{(\ca, \Mo)}_{R_M}.
\end{equation}
Define the functor $E:\End_\ca(\ca)\to \ca$, $E(F)=F(\uno)$. This functor is an equivalence of categories. Since for any $M\in \Mo$ we have that $E({}^* R_M \circ R_M)=\uhom(M,M)$, Proposition \ref{properties-end} (i) implies that there exists an isomorphism 
$$\widetilde{h}:  \int_{M\in \Mo} \uhom(M,M)\to E\big(  \int_{M\in \Mo}  {}^* R_M \circ R_M\big),$$
such that
\begin{equation}\label{tilde-h} E(\eta_M) \widetilde{h}= \pi^{\Mo}_M.
\end{equation}
Recall from Section \ref{SubSection:character algebra} the definition of the dinatural transformations $\pi^{\Mo}$.

Define $\phi:\ele(\ca,\Mo)(\uno)=\big(\int_{F\in \Fun_\ca(\ca, \Mo)} {}^* F\circ F\big)(\uno)\to  \int_{M\in \Mo} \uhom(M,M),$ as $\phi= (\widetilde{h})^{-1} E(h)$. Using \eqref{diagram-h-caseM} and \eqref{tilde-h} we get that
for any $M\in \Mo$
\begin{equation}\label{phi-on-M} \pi^\Mo_M\phi=E(\pi^{(\ca,\Mo)}_{R_M})=(\pi^{(\ca,\Mo)}_{R_M})_\uno.
\end{equation}
Let us prove that $\phi$ is a morphism in the center $Z(\ca)$. We need to show that for any $X\in\ca$
\begin{align}\label{phi-in-center} \sigma^{\Mo}_X(\phi\ot\id_X)= (\id_X\ot \phi) \alpha^{\sigma}_X.
\end{align}
Here $ \sigma^{\Mo}$ is the half-braiding of the algebra $\cha_\Mo$, as defined in Section \ref{SubSection:character algebra}. Recall from the proof of Theorem \ref{equivalence-centers} that 
$$\Phi( (\ele(\No,\Mo)_\No), \sigma)=(\ele(\ca,\Mo)(\uno), \alpha^\sigma)\in Z(\ca),$$
where $\alpha^\sigma $ is the half-braiding of $\ele(\ca,\Mo)(\uno)$, and it is defined by equation \eqref{half-brading-phi}, which in this case is
$$\alpha^{\sigma}_X: \ele(\ca,\Mo)(\uno)\ot X\to X\ot \ele(\ca,\Mo)(\uno),$$ 
$$ \alpha^{\sigma}_X=  c^{\ele(\ca,\Mo)}_{X,\uno}(\sigma_{R_X})^{-1}_\uno,$$
where $ c^{\ele(\ca,\Mo)}$ is the module structure of the module functor $\ele(\ca,\Mo)$, and $\sigma$ is the half-braiding in $Z(\camod)$ of the object $\ele(\No,\Mo)_\No$.

Using the universal property of the end, equation \eqref{phi-in-center} is equivalent to
\begin{align}\label{phi-in-center1} (\id_X\ot \pi^\Mo_M)\sigma^{\Mo}_X(\phi\ot\id_X)=(\id_X\ot \pi^\Mo_M) (\id_X\ot \phi) \alpha^{\sigma}_X,
\end{align}
for any $X\in \ca, M\in \Mo$. Using \eqref{half-braidig-ch}, one gets that the left hand side of  \eqref{phi-in-center1} is equal to
\begin{align*}&= \mathfrak{a}_{X,M,M} \mathfrak{b}_{X, M, X \otb M} (\pi^\Mo_{X \otb M} \otimes \id_X)(\phi\ot\id_X)\\
&=\mathfrak{a}_{X,M,M} \mathfrak{b}_{X, M, X \otb M} (  (\pi^{(\ca,\Mo)}_{R_{X\otb M}})_\uno\otimes \id_X).
\end{align*}

\begin{claim}For any $X\in \Mo$, $M\in \Mo$ the following equations hold.
\begin{equation}\label{claim-thad-1}  \mathfrak{a}_{X,M,M}(\pi^{(\ca,\Mo)}_{R_M})_X=\big( \id_X\ot (\pi^{(\ca,\Mo)}_{R_M})_\uno\big) c^{\ele(\ca,\Mo)}_{X,\uno},
\end{equation}
\end{claim}
\begin{proof}[Proof of Claim]  The functor ${}^* R_M \circ R_M:\ca\to \ca $ is a $\ca$-module functor. Since both ${}^* R_M, R_M$ are module functors, the  module structure of the composition is, according to \eqref{modfunctor-comp}, given by $d_{X,Y}:{}^* R_M \circ R_M(X\ot Y)\to X\ot R_M \circ R_M( Y)$, where
$$ d_{X,Y}= \mathfrak{a}_{X,M,Y\otb M}.$$
The  natural transformation 
$  \pi^{(\ca,\Mo)}_{R_M}: \int_{F\in \Fun_\ca(\No, \Mo)} {}^* F\circ F  \to {}^* R_M \circ R_M $ is a natural module transformation, this means that it satisfies \eqref{modfunctor3}, which in this case is
$$d_{X,Y}  \big(\pi^{(\ca,\Mo)}_{R_M}  \big)_{X\ot Y} =
 \big(\id_X\ot  (\pi^{(\ca,\Mo)}_{R_M})_Y \big)c^{\ele(\ca,\Mo)}_{X,Y}.$$
Taking $Y=\uno$ we obtain \eqref{claim-thad-1}.
\end{proof}
Using \eqref{phi-on-M}, the right hand side of  \eqref{phi-in-center1} is equal to
\begin{align*}&= ( \id_X\ot (\pi^{(\ca,\Mo)}_{R_M})_\uno )c^{\ele(\ca,\Mo)}_{X,\uno}(\sigma_{R_X})^{-1}_\uno\\
&= \mathfrak{a}_{X,M,M}(\pi^{(\ca,\Mo)}_{R_M})_X(\sigma_{R_X})^{-1}_\uno
\end{align*}
The second equality follows from \eqref{claim-thad-1}. Since $\mathfrak{a}_{X,M,M}$ is an isomorphism, equation \eqref{phi-in-center1} is equivalent to
\begin{equation}\label{equivalent-eq1} \mathfrak{b}_{X, M, X \otb M} (  (\pi^{(\ca,\Mo)}_{R_{X\otb M}})_\uno\otimes \id_X)(\sigma_{R_X})_\uno =(\pi^{(\ca,\Mo)}_{R_M})_X
\end{equation}
Recall that the half-braiding $\sigma_{R_X}$ is defined using diagram  \ref{definition-diagram-braiding}. In this particular case, this diagram is

\begin{equation}\label{definition-diagram-braiding2}
\xymatrix@C=80pt@R=16pt{
	\ele(\ca,\Mo)\circ R_X 
	\ar[r]^{\pi^{(\ca,\Mo)}_{R_M} \circ \id_{R_X}}
	\ar[d]_{\sigma_{R_X}}
	& {}^*R_M\circ R_M\circ R_X\\
	R_X\circ \ele(\ca,\Mo) \ar[r]_{\id_{R_X}\circ \pi^{(\ca,\Mo)}_{R_M\circ R_X}}
	& R_X\circ {}^* R_X\circ {}^* R_M\circ R_M\circ R_X \ar[u]_{\ev_{R_X}\circ\id_{{}^*R_MR_M R_X}} ,
}
\end{equation}
for any $X\in \ca$, $M\in \Mo$. Recall that, in this diagram the isomorphism ${}^* (R_M\circ R_X)\to  R_{X^*}\circ {}^*R_M $ is omitted. This isomorphism is described in Lemma \ref{duals-in-modcat}. Diagram \eqref{definition-diagram-braiding2} evaluated in $\uno$ implies that
$$(\ev_{R_X})_{\uhom(M,X\otb M)} (\mathfrak{b}^1_{X,M,X\otb M}\ot\id_X)(\pi^{(\ca,\Mo)}_{R_{X\otb M}})_\uno\otimes \id_X)(\sigma_{R_X})_\uno =(\pi^{(\ca,\Mo)}_{R_M})_X.$$
This implies equation \eqref{equivalent-eq1}. Let us prove now that $\phi: \Phi(\adj_\Mo)\to \cha_\Mo $ is an algebra map. Let us denote by 
$$ m^\Mo: \adj_\Mo\ot \adj_\Mo\to \adj_\Mo$$
 and  $m: \cha_\Mo \ot \cha_\Mo \to \cha_\Mo$ the corresponding multiplication morphisms. The product of $\Phi(\adj_\Mo)$ is given by 
 \begin{equation*}
 \xymatrix{\Phi(\adj_\Mo)\ot \Phi(\adj_\Mo)\ar[rr]^{ \zeta^\Phi}
\ar[dr]^{}&& \Phi(\adj_\Mo\ot \adj_\Mo)
\ar[dl]_{\Phi(m^\Mo)}\\ &\Phi(\adj_\Mo )&}.
\end{equation*}
Recall from the proof of Theorem \ref{equivalence-centers} the definition of $\zeta^\Phi$. The map $\phi$ is an algebra morphism if and only if
\begin{equation}\label{phi-algebra-map0} m (\phi\ot \phi)=\phi\, \Phi(m^\Mo) \big(c^{\ele(\ca,\Mo)}_{\ele(\ca,\Mo)(\uno),\uno}\big)^{-1}.
\end{equation}
Let $M\in \Mo$. Applying $\pi^\Mo_M$ to the left hand side of \eqref{phi-algebra-map0} one gets 
\begin{align*} \pi^\Mo_M  m (\phi\ot \phi)&= \comp^\Mo_{M}\circ (\pi^\Mo_M\ot \pi^\Mo_M)(\phi\ot \phi)\\
&= \comp^\Mo_{M} ((\pi^{(\ca,\Mo)}_{R_M})_\uno\ot (\pi^{(\ca,\Mo)}_{R_M})_\uno).
\end{align*}
The second equality follows from \eqref{phi-on-M}. Now, applying $\pi^\Mo_M$ to the right hand side of \eqref{phi-algebra-map0} one gets 
\begin{align*} \pi^\Mo_M\phi\, \Phi(m^\Mo) \big(c^{\ele(\ca,\Mo)}_{\ele(\ca,\Mo)(\uno),\uno}\big)^{-1}&= (\pi^{(\ca,\Mo)}_{R_M})_\uno (m^\Mo_\ca)_\uno  \big(c^{\ele(\ca,\Mo)}_{\ele(\ca,\Mo)(\uno),\uno}\big)^{-1}
\end{align*}
\epf

As a direct consequence of Lemma \ref{iso-adj-0cell} and Theorem \ref{adjoint2cat=adjointS}, we have the following result.

\begin{cor} Assume that $\Mo$ and $\No$ are equivalent exact $\ca$-module categories. Then, the algebras $\cha_\Mo, \cha_\No$ are isomorphic.\qed
\end{cor}

\begin{teo}\label{adj-to2-equiva} Assume that $\Bc, \wbc$ are finite 2-categories.  Let $\Fc:\Bc\to \wbc$ be a biequivalence, and let $\widehat{\Fc}: Z(\Bc)\to Z(\wbc)$ be the associated monoidal equivalence given in Proposition \ref{center-2cat-monoid}. For simplicity we shall further assume that $\Fc$ is a 2-functor.  Hence, for any 0-cell $B\in \Bc^0$ there is an isomorphism 
$$\widehat{\Fc}(\adj_B)\simeq  \adj_{\Fc(B)}$$
as algebras  in the category $Z(\wbc)$.
\end{teo}
\pf  Let $\Gc:\wbc \to \Bc $ be the quasi-inverse of $\Fc$. Since 0-cells $B, \Gc(\Fc(B))$ are equivalent, then, by Lemma \ref{iso-adj-0cell} the algebras $\adj_B, \adj_{\Gc(\Fc(B))}$ are isomorphic. We will prove that there is an algebra isomorphism $\widehat{\Fc}(\adj_{\Gc(\Fc(B))})\simeq  \adj_{\Fc(B)}$

 Let $\tau: \Id \to \Fc\circ \Gc$, $\chi: \Fc\circ \Gc\to  \Id$ be a pair of pseudonatural equivalences, one the inverse of the other. Hence $\chi\circ \tau\sim \id_{\Id}$, and $\tau\circ \chi\sim \id_{\Fc\circ \Gc}$.
For any pair of 0-cells $C, D\in \wbc^0,$
$\chi^0_C\in \wbc(\Fc(\Gc(C),C)$, $\tau^0_C\in \wbc(C, \Fc(\Gc(C))$ are 1-cells, and for any 1-cell $Y\in \wbc(C,D)$
$$\chi_Y: Y\circ \chi^0_C\Longrightarrow \chi^0_D\circ \Fc(\Gc(Y)),$$
$$\tau_Y: \Fc(\Gc(Y))\circ \tau^0_C\Longrightarrow\tau^0_D\circ Y. $$
In particular, for any 0-cell $C\in \Bc$ we have that $\chi^0_C\circ \tau^0_C\simeq I$. Thus we can assume that ${}^*(\tau^0_C)=\chi^0_C$.
 
For any 0-cell $C\in \wbc^0$ define the functors
$$\Hc: \Bc(\Gc(C),\Gc(\Fc(B)))\to \wbc(C,\Fc(B)),$$
$$\Hc(X)=\chi^0_{\Fc(B)}\circ \Fc(X)\circ \tau^0_C,$$
$$\widetilde{\Hc}:\wbc(C,\Fc(B)) \to  \Bc(\Gc(C),B),$$
$$ \widetilde{\Hc}(Z)=\Gc(Z). $$
These functors are equivalences, one the quasi-inverse of the other. 
Define also the natural isomorphism 
$$ \alpha: \Hc  \widetilde{\Hc}\to \Id,$$
$$ \alpha_Z=\id_{\chi^0_{\Fc(B)}} \circ \tau_Z,$$
for any $Z\in \wbc(C,\Fc(B))$. Let $C\in \wbc^0$ be a 0-cell, then, using the definition of $\widehat{\Fc}$ given in the proof of  Proposition \ref{center-2cat-monoid},  we have that
\begin{align*}
 \widehat{\Fc}(\adj_{\Gc(\Fc(B))})_C&=\chi^0_C\circ \Fc(   \ele(\Gc(C),\Gc(\Fc(B))))\circ \tau^0_C \\
 &= \int_{X\in \Bc(\Gc(C),\Gc(\Fc(B)))} {}^* \Hc(X)\circ \Hc(X).
\end{align*} 
Here we used  that $\circ$ is biexact and applied Proposition \ref{properties-end} (i). Also
$$( \adj_{\Fc(B)})_C=  \int_{Y\in \wbc(C,\Fc(B))}  {}^*Y\circ Y.$$
Let $$\widetilde{\pi}^{(\Fc(B), C)}_Y: ( \adj_{\Fc(B)})_C\to  {}^*Y\circ Y$$ and $$\lambda_X: \int_{X\in \Bc(\Gc(C),\Gc(\Fc(B)))}  {}^* \Hc(X)\circ \Hc(X)\to {}^*X\circ X$$
be the associated dinatural transformations. As a space saving measure we will write $\widetilde{\pi}_Y=\widetilde{\pi}^{(\Fc(B), C)}_Y$.

Since the functor $\Hc:\Bc(\Gc(C), \Gc(\Fc(B)))\to \wbc(C,\Fc(B))$ is an equivalence of categories, using (the proof of) Proposition \ref{properties-end} (ii), we get that there is an isomorphism 
$$ h^C: \widehat{\Fc}(\adj_{\Gc(\Fc(B))})_C\to ( \adj_{\Fc(B)})_C$$
such that
\begin{equation}\label{h-iso-def4}   \big({}^*(\alpha^{-1}_Z)\circ \alpha_Z \big) (\id_{\chi^0_C}\circ \Fc(\lambda_{\widetilde{\Hc}(Z)})\circ \id_{\tau^0_C}) = \widetilde{\pi}_Z \,h^C,
\end{equation}
for any $Z\in \wbc(C,\Fc(B))$. Let us prove that $h:\widehat{\Fc}(\adj_{\Gc(\Fc(B))})\to \adj_{\Fc(B)}$ defines an algebra map in the center $Z(\wbc)$. Let  $\sigma$ and $\widetilde{\sigma}$ be the half-braidings of $\adj_{\Gc(\Fc(B))}$ and $\adj_{\Fc(B)}$ respectively. To prove that $h$ is a morphism in the center, we need  to show that equation
\begin{equation}\label{h-is-braided}   (\widetilde{\sigma}^B_Z)^{-1} (\id_Z\circ h^C)= (h^D\circ \id_Z) \widehat{\Fc}( (\sigma^B)^{-1})_Z
\end{equation}
is satisfied  for any 1-cell $Z\in \wbc(C,D)$. Recall that the definition of $\widehat{\Fc}( \sigma^{-1}) $ is given in the proof of Proposition \ref{center-2cat-monoid}. To prove  equation \eqref{h-is-braided} it is enough to show that
\begin{equation}\label{h-is-braided2}  
 (\widetilde{\pi}_Y\circ \id_Z)(\widetilde{\sigma}_Z)^{-1} (\id_Z\circ h^C)=  (\widetilde{\pi}_Y\circ \id_Z)(h^D\circ \id_Z) \widehat{\Fc}( \sigma^{-1})_Z.
\end{equation}
for any 1-cell $Y\in \wbc(D,\Fc(B))$. Using \eqref{definition-diagram-braiding}, we obtain that the left hand side of \eqref{h-is-braided2}  is equal to
\begin{align*} &=(\ev_Z\circ \id_{{}^*Y\circ Y\circ Z}) (\id_Z\circ \widetilde{\pi}_{Y\circ Z})(\id_Z\circ h^C)\\
&=(\ev_Z\circ \id_{{}^*Y\circ Y\circ Z}) (\id_Z\circ {}^*(\tau^{-1}_{Y\circ Z})\circ \tau_{Y\circ Z})(\id\circ \Fc(\lambda_{\Gc(Y)\circ \Gc(Z)})\circ \id).
\end{align*}
The second equality follows from \eqref{h-iso-def4}. Next, as a space saving measure, we will denote $ E_C=  \int_{X\in \Bc(\Gc(C),\Gc(\Fc(B)))}  {}^* \Hc(X)\circ \Hc(X)$. Using \eqref{h-iso-def4} we get that the right hand side of \eqref{h-is-braided2}  is equal to
\begin{align*} &= \big({}^*(\alpha^{-1}_Y)\circ \alpha_Y\circ\id_Z) \big)(\id_{\chi^0_D}\circ \Fc(\lambda_{\widetilde{\Hc}(Y)})\circ\id_{\tau^0_D\circ Z}) \widehat{\Fc}( (\sigma^B)^{-1})_Z\\
&=({}^*(\tau^{-1}_Y)\circ \tau_Y\circ\id_Z )(\id_{\chi^0_D}\circ \Fc(\lambda_{\widetilde{\Hc}(Y)})\circ\id_{\tau^0_D\circ Z}) 
(\id_{\chi^0_D\circ \Fc(E_D)} \circ \tau_Z)\\
& (\id\circ \Fc(\sigma^{-1}_{\Gc(Y)})\circ \id)(\chi_Z\circ \id_{\Fc(E_C)\circ \tau^0_C})\\
&= ({}^*(\tau^{-1}_Y)\circ \tau_Y \circ\id_Z) (\id_{\chi^0_D\circ {}^*\Fc(
\Gc(Z))\circ \Fc(\Gc(Z))} \circ \tau_Z)\\ &(\id_{\chi^0_D}\circ \Fc(\lambda_{\widetilde{\Hc}(Y)})\circ \id_{\Fc(\Gc(Z))\circ \tau^0_C}) 
(\id\circ \Fc(\sigma^{-1}_{\Gc(Y)})\circ \id)(\chi_Z\circ \id_{\Fc(E_C)\circ \tau^0_C})\\
&=({}^*(\tau^{-1}_Y)\circ \tau_Y\circ\id_Z ) (\id_{\chi^0_D\circ {}^*\Fc(
\Gc(Z))\circ \Fc(\Gc(Z))} \circ \tau_Z)\\
& (\id\circ \Fc((\lambda_{\Gc(Y) }\circ \id) \sigma^{-1}_{\Gc(Y)}) \circ \id ) (\chi_Z\circ \id_{\Fc(E_C)\circ \tau^0_C})\\
&=({}^*(\tau^{-1}_Y)\circ \tau_Y\circ\id_Z ) (\id \circ \tau_Z)  (\id \circ \Fc(\ev_{\Gc(Z)} \circ \id ) (\id\circ \Fc(\lambda_{\Gc(Y)\circ \Gc(Z)} ) \circ \id)\\
& (\chi_Z\circ \id_{\Fc(E_C)\circ \tau^0_C})\\
&=({}^*(\tau^{-1}_Y)\circ \tau_Y \circ\id_Z) (\id \circ \tau_Z)  \big( \id_{\chi^0_D} \circ \Fc(\Gc(\ev_Z)) (\chi_Z\circ \id_{{}^*\Fc(\Gc(Z))})\circ \id\big)\\
& (\id_{Z\circ \chi^0_C}\circ \Fc(\lambda_{\Gc(Y)\circ \Gc(Z)} ) \circ \id_{\tau^0_C})
\end{align*}
The second equality follows from the definition of $\widehat{\Fc}( \sigma^{-1}) $, and the  fifth equality follows from \eqref{definition-diagram-braiding}. The naturality of $\chi$ implies that for any 1-cell $Z\in \wbc(C,D)$
$$ (\id_{\chi^0_D}\circ \Fc(\Gc(\ev_Z))) \chi_{Z\circ {}^*Z}=\ev_Z\circ \id_{\chi^0_D}.$$
Using \eqref{pseudonat-def1} this equation implies that
\begin{align}\label{nat-chi-ev} \begin{split}(\id_{\chi^0_D}\circ \Fc(\Gc(\ev_Z))) (\chi_Z\circ  \id_{{}^*\Fc(\Gc(Z))})&=(\ev_Z\circ \id_{\chi^0_D})(\id_Z\circ \chi^{-1}_{{}^*Z})\\
&=(\ev_Z\circ \id_{\chi^0_D})(\id_Z\circ {}^*(\tau_{Z})^{-1}).
\end{split}
\end{align} 
Note that in the second equality we have used Lemma \ref{dual-of-pseudonat-equi}. Now, continuing with the  right hand side of \eqref{h-is-braided2}, and using \eqref{nat-chi-ev} we get that it is equal to
\begin{align*}
=({}^*(\tau^{-1}_Y)\circ \tau_Y \circ\id_Z )& (\id \circ \tau_Z) (\ev_Z\circ \id))(\id_Z\circ  {}^*(\tau^{-1}_{Z})\circ \id)\\
& (\id_{Z\circ \chi^0_C}\circ \Fc(\lambda_{\Gc(Y)\circ \Gc(Z)} ) \circ \id_{\tau^0_C}).
\end{align*}
Using  \eqref{pseudonat-def1} for $\tau_{Y\circ Z}$ we see that both sides are equal, 
and $h$ defines a morphism in the center. Let us prove now, that $h$ defines an algebra morphism. Let $\widetilde{m}: \adj_{\Fc(B)}\ot \adj_{\Fc(B)}\to \adj_{\Fc(B)}$ be the multiplication map. Also, if $m: \adj_{\Gc(\Fc(B))}\ot \adj_{\Gc(\Fc(B))}\to  \adj_{\Gc(\Fc(B))}$  is the product of the adjoint algebra, then the product for $\widehat{\Fc}( \adj_{\Gc(\Fc(B))})$ is 
$$\id_{\chi^0_C}\circ \Fc(m_C)\circ \id_{\tau^0_C}, $$
for any 0-cell $C\in \wbc^0$. Hence, we need to show that
\begin{equation} h^C (\id_{\chi^0_C}\circ \Fc(m_C)\circ \id_{\tau^0_C})= \widetilde{m}_C (h^C\circ h^C),
\end{equation}
for any 0-cell $C\in \wbc^0$.  For this, it is enough to prove that
\begin{equation}\label{h-is-algebra}  \widetilde{\pi}_Z  h^C (\id_{\chi^0_C}\circ \Fc(m_C)\circ \id_{\tau^0_C})=\widetilde{\pi}_Z \widetilde{m}_C (h^C\circ h^C),
\end{equation}
for any 1-cell $Z\in \wbc(C,D)$. Using \eqref{h-iso-def4}, we get that the left hand side of \eqref{h-is-algebra} is equal to 
\begin{align*}  &=\big({}^*(\alpha^{-1}_Z)\circ \alpha_Z \big) (\id_{\chi^0_C}\circ \Fc(\lambda_{\widetilde{\Hc}(Z)}m_C)\circ \id_{\tau^0_C}) \\
&=\big({}^*(\alpha^{-1}_Z)\circ \alpha_Z \big) 
(\id_{\chi^0_C}\circ \Fc((\id_{{}^*\Gc(Z)}\circ\ev_{\Gc(Z)}\circ \id_{\Gc(Z)} ) (\lambda_{\Gc(Z)}\circ \lambda_{\Gc(Z)}) )\circ \id_{\tau^0_C})\\
&=\big({}^*(\alpha^{-1}_Z)\circ \alpha_Z \big)  (\id\circ \Fc(\id_{{}^*\Gc(Z)}\circ\ev_{\Gc(Z)}\circ \id_{\Gc(Z)} ) \Fc(\lambda_{\Gc(Z)}\circ \lambda_{\Gc(Z)}) )\circ \id).
\end{align*}
The second equality follows from the definition of the product of the adjoint algebra given in \eqref{product-ch-2cat}. Also, using  \eqref{product-ch-2cat} we get that the right hand side of \eqref{h-is-algebra} is equal to 
 \begin{align*}  &=(\id_{{}^*Z}\circ\ev_Z\circ\id_Z )( \widetilde{\pi}_Z h^C\circ \widetilde{\pi}_Z h^C) \\
 &= (\id_{{}^*Z}\circ\ev_Z\circ\id_Z ) \big({}^*(\alpha^{-1}_Z)\circ \alpha_Z\circ {}^*(\alpha^{-1}_Z)\circ \alpha_Z\big) (\id\circ\Fc(\lambda_{\Gc(Z)}\circ \lambda_{\Gc(Z)}) )\circ \id).
 \end{align*}
The second equality follows from  \eqref{h-iso-def4}. It follows  from \eqref{nat-chi-ev} that both sides are equal.  \epf

Applying Theorem \ref{adj-to2-equiva} to the 2-category of representations of a tensor category, and using Theorem \ref{adjoint2cat=adjointS}, we get the next result.

\begin{cor} Let $\ca, \Do$ be finite tensor categories. Assume that  $\No$ is an invertible $(\D,\ca)$-bimodule category, and $\Mo$ be an indecomposable left exact $\ca$-module. There is an isomorphism of algebras 
$$\theta(\cha_\Mo)\simeq \cha_{\,\Fun_\ca(\No,\Mo)}. $$
Here $\theta:Z(\ca)\to Z(\Do)$ is the monoidal equivalence presented in \eqref{theta-equival-center}.\qed
\end{cor}

\subsection*{Acknowledgements}
This work was partially supported by CONICET and
Secyt (UNC), Argentina. We thank the referee for his/her comments, that improved the presentation of the paper.

\end{document}